\documentclass{aims}
\usepackage{amsmath}
  \usepackage{paralist}
  \usepackage{graphics} 
  \usepackage{epsfig} 
  \usepackage{setspace}
\usepackage{graphicx}  \usepackage{epstopdf}
 \usepackage[colorlinks=true]{hyperref}
\hypersetup{urlcolor=blue, citecolor=red}
\allowdisplaybreaks[4]
\def\currentvolume{X}
 \def\currentissue{X}
  \def\currentyear{200X}
   \def\currentmonth{XX}
    \def\ppages{X--XX}
     \def\DOI{10.3934/xx.xx.xx.xx}
 
\newtheorem{theorem}{Theorem}[section]
\newtheorem{exmp}[theorem]{Example}
\newtheorem{hypothesis}{Hypothesis}[section]

\newtheorem{lemma}[theorem]{Lemma}

\theoremstyle{definition}
\newtheorem{definition}[theorem]{Definition}
\newtheorem{remark}{Remark}

\makeatletter
\def\@textbottom{\vskip \z@ \@plus 200pt}
\let\@texttop\relax
\makeatother

\title[SECOND-ORDER STOCHASTIC MAXIMUM PRINCIPLE ] 
      {A SECOND-ORDER STOCHASTIC MAXIMUM PRINCIPLE FOR GENERALIZED MEAN-FIELD SINGULAR CONTROL PROBLEM}

\author[Hancheng Guo and Jie Xiong]{}


 \email{guohancheng1989@gmail.com}
 \email{jiexiong@umac.mo}

\thanks{Research supported partially by FDCT 025/2016/A1.}

\thanks{$^*$ Corresponding author: Hancheng Guo}
\begin{document}
\maketitle

\centerline{\scshape Hancheng Guo$^*$}
\medskip
{\footnotesize
 \centerline {Department of Mathematics, Faculty of Science and Technology
}
   \centerline{University of Macau}
   \centerline{ Macau, 999078, P. R. China}
} 

\medskip

\centerline{\scshape Jie Xiong
}
\medskip
{\footnotesize
 \centerline{Department of Mathematics, Faculty of Science and Technology
}
   \centerline{University of Macau}
   \centerline{ Macau, 999078, P. R. China}
}

\bigskip

 \centerline{(Communicated by the associate editor name)}
\begin{spacing}{1.0}

\begin{abstract}
In this paper, we study the generalized mean-field stochastic control problem when the usual stochastic maximum principle (SMP) is not applicable due to the singularity of the Hamiltonian function. In this case, we derive a second order SMP. We introduce the adjoint process by the generalized mean-field backward stochastic differential equation. The keys in the proofs are the expansion of the cost functional in terms of a perturbation parameter, and the use of the range theorem for vector-valued measures.
\end{abstract}


\section{Introduction}

We consider the following optimal stochastic control problem of mean-field type with the state equation
\begin{equation}\label{eq210c}
\left\{
\begin{array}{l}
dX_t=b(t,X_t,P_ {{X_t}},v_t) dt
    +\sigma(t,X_t,P_ {{X_t}}) dB_t,\\
X_0=x,
\end{array}
\right.
\end{equation}
and the cost functional
\begin{equation}\label{eq211c}
\begin{split}                                                                                                                                                                                                                                                                                                                                                                                                                                                                                                                                                                                                                   
  \begin{aligned}
J(v)=\mathbb{E}\left\{ \int^T_0 h(t,X_t,P_ {{X_t}},v_t)dt+\Phi(X_T,P_ {{X_T}})\right\},
\end{aligned}
\end{split}
 \end{equation}
 where $P_\xi$ denotes the law of the random variable $\xi$.

The agent wishes to minimize his cost functional, namely, an admissible control $u\in \mathcal{U}$ is said to be optimal if
$$J(u)=\min_{v\in \mathcal{U}}J(v).$$
where $\mathcal{U}$ is the set of all admissible controls to be defined later in Section 3.

 About stochastic maximum principle (SMP), some pioneering works have been done by  Pontryagin et al. \cite{P1962}. They obtained Pontryagin's maximum principle by using ``spike variation''.  Kushner (\cite{K1965}, \cite{K1972}) studied the SMP in the framework when the diffusion coefficient does not depend on the control variable, and the cost functional consists of terminal cost only. Haussmann \cite{H1986} gave a version of SMP when the diffusion of the state does not depend on the control variable. Arkin and Saksonov \cite{AS1979}, Bensoussan  \cite{B1982} and Bismut \cite{B1978}, proved different versions of SMP under various setups. An SMP was obtained by Peng \cite{P1990} in 1990. In that paper, first and second order variational inequalities are introduced, when the control domain need not to be convex, and the diffusion coefficient contains the control variable.

 Pardoux and Peng \cite{PP1990} introduced non-linear backward stochastic differential equations (BSDE) in 1990. They showed that under appropriate assumptions, BSDE admits an unique adapted solution, and the associated comparison theorem holds. Buckdahn et al \cite{BLP2009} obtained mean-field BSDE in a natural way as the limit of some high dimensional system of forward and backward stochastic differential equations. Li \cite{L2012} studied SMP for mean-filed controls when the domain of the control is assumed to be convex. Under some additional assumptions, both necessary and sufficient conditions for the optimality of a control were proved. Buckdahn et al \cite{BLP2014} studied generalized mean-field stochastic differential equations and the associated partial differential equations (PDEs). ``Generalized" means that the coefficients depend on both the state process and its law. They proved that under appropriate regularity conditions on the coefficients, the SDE has a unique classical solution.  Buckdahn et al. \cite{BLM2016} obtained SMP for generalized mean-field system in 2016.

 Sometimes, the Hamiltonian function becomes constant in the control variable, as we will see in the next example, which makes the aforementioned SMP not applicable.
\begin{exmp}\label{em11}
Consider the control problem with state equation:
\begin{equation}\label{eq0210a}
\left\{
\begin{array}{ccl}
dX^v_t&=&v_tdt+\left\{(X^v_t-1)+\mathbb{E}\left[(X^v_t-1)\right]\right\}dB_t, \ v\in {U}:=\{-1,0,1\},\\
X^v_0&=&1,
  \end{array}
  \right.
\end{equation}
and cost functional:
$$J(v)=\frac{1}{2}\mathbb{E}\big\{(X^v_T-1)+\tilde{\mathbb{E}}\big[(X^v_T-1)\big]\big\}^2.$$
\indent For the control $u_t\equiv 0$, $X^u_t\equiv 1$ is the unique solution of (\ref{eq0210a}). It is clear that $J(u)=0$, and hence, $u$ is an optimal control. On the other hand, the first order adjoint processes satisfy the following equation:
\begin{equation}
\left\{
\begin{array}{ccl}
dp_t&=&\left\{q_t+\mathbb{E}\left[{q_t}\right] \right\}dt-q_tdB_t \\
p_T&=&0.
  \end{array}
  \right.
\end{equation}
 Clearly $(p_t,q_t)\equiv (0,0) $ is the solution. Therefore
$$H(t,X^u_t,P_{X^u_t},p_t,q_t,v)\equiv 0,\ \ \ v\in {U}.$$
which makes the SMP useless in charactering the optimal control $u_t=0.$
\qed
\end{exmp}

 Now, we discuss singular optimal stochastic controls defined as follows.
\begin{definition}
An admissible control $\tilde{u}(\cdot)$ is singular on region $V$ if $V\subset U$ is of positive measure and for $a.e.\ t\in [0,T]$ and $v\in V$, we have for any $v\in V$,
 \begin{eqnarray}
  H(t,X^{\tilde{u}}_t,P_{X^{\tilde{u}}_t},p_t^ {\tilde{u}},q_t^ {\tilde{u}},\tilde{u}_t)=H(t,X^{\tilde{u}}_t,P_{X^{\tilde{u}}_t},p_t^ {\tilde{u}},q_t^ {\tilde{u}},v), \quad a.s.
  \end{eqnarray}
\end{definition}

As we have seen in last example, the SMP is not very useful under singular control. Our goal is to derive further necessary condition for optimality. We shall call the original SMP as the first order SMP while the one we will derive as the second  order one.

   For second-order SMP of singular control problems, Bell \cite{BJ1975}, Gabasov \cite{GK1972}, Kazemi-Dehkordi \cite{KD1984}, Krener \cite{K1977}, Mizukami and Wu \cite{MW1992} devoted themselves to the deterministic case. Lu \cite{L2016} interested in second order necessary conditions for  stochastic evolution system. Tang \cite{T2010} studied the singular optimal control problem for stochastic system with state equation
\begin{equation}
\left\{
\begin{array}{l}
dX_t=b(t,X_t,v_t) dt
    +\sigma(t,X_t) dB_t,\\
X_0=x.
\end{array}
\right.
\end{equation}
and the cost functional
\begin{equation}
\begin{split}
  \begin{aligned}
J(v)=\mathbb{E}\{ \int^T_0 h(t,X_t,v_t)dt+\Phi(X_T)\},
\end{aligned}
\end{split}
 \end{equation}
By applying spike variation and vector-value measure theory, a second-order maximum principle is presented which involves the second-order adjoint process.

  In this paper, we study the case when the state equation and the cost functional are in generalized mean-field form. The rest of this paper is organized as follows: In Section 2, we introduce the preliminaries about the generalized mean-field BSDEs. In Section 3, we set up the formulation of the singular optimal stochastic control problem and state the main result of the paper. Section 4 is devoted to the study of the impact of the control actions on the state and the cost functional by using Taylor's expansion. In that section, we also present some estimations about the state. In Section 5, the method in Section 4 is reused for the expansion of the cost functional with respect to the control variable.  Sections 6 is devoted to the proof of the second order stochastic maximum principle.

\section{Preliminaries}

\indent In this section, for the convenience of the reader, we state some results of Buckdahn et al. \cite{BLP2014} without proofs.

 Let $\mathcal{P}_2(\mathbb{R}^n)$  be the collection of all square integrable probability measures over $(\mathbb{R}^n,\mathcal{B}(\mathbb{R}^n))$, endowed with the 2-Wasserstein metric $W_2$, which is defined as
$$W_2(P_{\mu },P_{\nu})=\inf\left\{ \left(\mathbb{E}[|\mu'-\nu'|^2]\right)^{\frac{1}{2}}\right \},$$
for all $\mu',\nu' \in L^2(\mathcal{F}_0;\mathbb{R}^d)$ with $P_{\mu'}=P_{\mu}, \ P_{\nu'}=P_{\nu}.$ Denote by $L^2(\mathcal{F};\mathbb{R}^n)$ the collection of all $\mathbb{R}^n$-valued square integrable random variables. The following definition is taken from Cardaliaguet \cite{C 2013}.
\begin{definition}
A function $f: \mathcal{P}_2(\mathbb{R}^n)\longrightarrow \mathbb{R} $ is said to be differentiable in $\mu \in \mathcal{P}_2(\mathbb{R}^n)$ if, the function $\tilde{f}: L^2(\mathcal{F};\mathbb{R}^n)\longrightarrow \mathbb{R}$ given by $\tilde{f}(\mathfrak{v})=f(P_{\mathfrak{v}})$ is differentiable (in Fr\'{e}chet sense) at $\mathfrak{v}_0$, defined by $P_{\mathfrak{v}_0}=\mu$, i.e. there exists a linear continuous mapping $D\tilde{f}(\mathfrak{v}_0):L^2( \mathcal{F} ;\mathbb{R}^n)\longrightarrow \mathbb{R},$ such that
$$\tilde{f}(\mathfrak{v}_0+\eta)-\tilde{f}(\mathfrak{v}_0)=D\tilde{f}(\mathfrak{v}_0)(\eta)+o(|\eta|_{L^2}),$$
with $|\eta|_{L^2}\longrightarrow 0$ for $\eta \in L^2( \mathcal{F} ;\mathbb{R}^n). $
\end{definition}

 According to the Riesz representation theorem, there exists a unique random variable $\theta_0\in L^2( \mathcal{F} ;\mathbb{R}^n)$ such that $D \tilde{f}(\mathfrak{v}_0)(\eta)=(\theta_0,\eta)_{L^2}=\mathbb{E}[\theta_0\eta]$, for all $\eta\in  L^2( \mathcal{F} ;\mathbb{R}^n).$ In \cite{C 2013} it has been proved that there is a Borel function $h_0:\mathbb{R}^n\longrightarrow\mathbb{R}^n$ such that $\theta_0=h_0(\mathfrak{v}_0)$ a.s. Then,
$$f(P_{\mathfrak{v}})-f(P_{\mathfrak{v}_0})
=\mathbb{E}\left[h_0(\mathfrak{v}_0)(\mathfrak{v}-\mathfrak{v}_0)\right]
+o(|\mathfrak{v}-\mathfrak{v}_0|_{L^2}),\qquad \mathfrak{v}\in L^2( \mathcal{F} ;\mathbb{R}^n). $$

 We call $\partial_{\mu} f( P_{\mathfrak{v}_0},y):=h_0(y), \ y\in \mathbb{R}^n $, the derivative of $f: \mathcal{P}_2(\mathbb{R}^n)\longrightarrow \mathbb{R}^n$ at $P_{\mathfrak{v}_0}.$ Note that  $\partial_{\mu} f( P_{\mathfrak{v}_0},y)$ is $P_{\mathfrak{v}_0}(dy) $-$a.s.$ uniquely determined. \\
\indent For mean-field type SDE and BSDE, we introduce the following notations. Let $(\Omega',{\mathcal{F}}',P')$ be a copy of the probability space $(\Omega,{\mathcal{F}},P)$. For each random variable $\xi$ over $(\Omega,{\mathcal{F}},P)$ we denote by $\xi'$ a copy of $\xi$ defined over $(\Omega',{\mathcal{F}}',P')$. $\mathbb{E}'[\cdot]=\int_{\Omega'}(\cdot)dP'$ acts only over the variables  $\omega'$.

\begin{definition}
We say that $f\in C^{1,1}_b(\mathcal{P}_2(\mathbb{R}^d))$  (continuously differentiable over $\mathcal{P}_2(\mathbb{R}^d)$ with Lipschitz-continuous bounded derivative), if for all $\mathfrak{v}\in L^2(\mathcal{F},\mathbb{R}^d)$, there exists a $P_{\mathfrak{v}}$-modification of $\partial_{\mu }f(P_{\mathfrak{v}},\cdot)$, again denote by $\partial_{\mu }f(P_{\mathfrak{v}},\cdot)$, such that $\partial_{\mu }f:\mathcal{P}_2(\mathbb{R}^d)\times \mathbb{R}^d\longrightarrow \mathbb{R}^d$ is bounded and Lipschitz continuous, i.e., there is a real constant $C$ such that
\begin{equation}
\begin{split}
  \begin{aligned}
i)\ &|\partial_{\mu }f(\mu ,x)|\leq C,\qquad\forall \mu \in\mathcal{P}_2( \mathbb{R}^d), x\in \mathbb{R}^d,&\\
ii)\ &|\partial_{\mu }f(\mu ,x)-\partial_{\mu }f(\mu' ,x')|\leq C\left(W_2(\mu,\mu')+|x-x'|\right),&\\
&\forall \mu, \mu' \in\mathcal{P}_2( \mathbb{R}^d), x,x'\in \mathbb{R}^d;&
 \end{aligned}
  \end{split}
\end{equation}
we call this function $\partial_{\mu }f$ the derivative of $f$.
\end{definition}

\indent Given $f\in C^{1,1}_b(\mathcal{P}_2(\mathbb{R}^d))$, and $y\in \mathbb{R}^d$, the question of the differentiability of its components $(\partial_{\mu }f)_j(\cdot , y): \mathcal{P}_2(\mathbb{R}^d)\rightarrow \mathbb{R}, \ 1\leq j \leq d,$ raises. This can be discussed in the same way as the first order derivative $\partial_{\mu}f$ above. If $(\partial_{\mu }f)_j(\cdot , y): \mathcal{P}_2(\mathbb{R}^d)\rightarrow \mathbb{R}$ belongs to $C^{1,1}_b(\mathcal{P}_2(\mathbb{R}^d))$, we have that its derivative $\partial_{\mu}((\partial_{\mu }f)_j(\cdot , y))(\cdot , \cdot): \mathcal{P}_2(\mathbb{R}^d)\times \mathbb{R}^d\rightarrow \mathbb{R}^d$ is a Lipschitz-continuous function. Then
$$\partial^2_{\mu}f(\mu, x,y):=\left(\partial_{\mu}\left(\left(\partial_{\mu }f)_j\left(\cdot , y\right)\right)(\mu , x\right)\right)_{1\leq j \leq d},\ \ (\mu ,x,y)\in \mathcal{P}_2(\mathbb{R}^d)\times \mathbb{R}^d\times \mathbb{R}^d,$$
defines a function $\partial^2_{\mu}f:\mathcal{P}_2(\mathbb{R}^d)\times \mathbb{R}^d\times \mathbb{R}^d\rightarrow \mathbb{R}^d\otimes\mathbb{R}^d.$

\begin{definition}
We say that $f\in C^{2,1}_b(\mathcal{P}_2(\mathbb{R}^d))$, if $f\in C^{1,1}_b(\mathcal{P}_2(\mathbb{R}^d))$ and \\
i) $(\partial_{\mu }f)_j(\cdot , y)\in C^{1,1}_b(\mathcal{P}_2(\mathbb{R}^d))$, for all $y\in \mathbb{R}^d, \ 1\leq j \leq d$, and $\partial^2_{\mu}f: \mathcal{P}_2(\mathbb{R}^d)\times \mathbb{R}^d\times \mathbb{R}^d\rightarrow \mathbb{R}^d\otimes\mathbb{R}^d$ is bounded and Lipschitz-continuous;\\
ii)  $(\partial_{\mu }f)(\mu , \cdot): \mathbb{R}^d\rightarrow \mathbb{R}^d$ is differentiable for every $\mu \in \mathcal{P}_2(\mathbb{R}^d),$ and its derivative $\partial_y \partial_{\mu}f:\mathcal{P}_2(\mathbb{R}^d)\times \mathbb{R}^d\rightarrow \mathbb{R}^d\otimes\mathbb{R}^d$ is bounded and Lipschitz-continuous.
\end{definition}

\begin{exmp}
For twice continuously differentiable functions $h:\mathbb{R}^d\rightarrow \mathbb{R}$ and $g:\mathbb{R}\rightarrow \mathbb{R}$ with bounded derivatives. Consider $f(P_{\mathfrak{v}}):=g(\mathbb{E}[h(\mathfrak{v})]),$ $\mathfrak{v}\in L^2(\mathcal{F};\mathbb{R}^d).$ Then, given any $\mathfrak{v}_0\in  L^2(\mathcal{F};\mathbb{R}^d),$ $\tilde{f}(\mathfrak{v}):=f(P_{\mathfrak{v}})=g(\mathbb{E}[h(\mathfrak{v})])$ is Fr\'{e}chet differentiable in $\mathfrak{v}_0$, and
\begin{eqnarray*}
\tilde{f}(\mathfrak{v}_0+\eta)-\tilde{f}(\mathfrak{v}_0)&=&\int^1_0 g'(\mathbb{E}[h(\mathfrak{v}_0+s\eta)])\mathbb{E}[h'(\mathfrak{v}_0+s\eta)\eta]ds\\
&&=g'(\mathbb{E}[h(\mathfrak{v}_0)])\mathbb{E}[h'(\mathfrak{v}_0)\eta]+o(|\eta|_{L^2})\\
&&=\mathbb{E}[g'(\mathbb{E}[h(\mathfrak{v}_0)])h'(\mathfrak{v}_0)\eta]+o(|\eta|_{L^2}).
\end{eqnarray*}
So, $D\tilde{f}(\mathfrak{v}_0)(\eta)=\mathbb{E}[g'(\mathbb{E}[h(\mathfrak{v}_0)])h'(\mathfrak{v}_0)\eta],$ $\eta\in L^2(\mathcal{F};\mathbb{R}^d), i.e.,$
$$\partial_{\mu}f(P_{\mathfrak{v}_0},y)=g'(\mathbb{E}[h(\mathfrak{v}_0)])(\partial_y h)(y), y\in \mathbb{R}^d.$$
Similarly, we see that
\begin{eqnarray*}
\partial_{\mu}^2f(P_{\mathfrak{v}_0},x,y)=g''(\mathbb{E}[h(\mathfrak{v}_0)])(\partial_xh)(x)\times(\partial_yh)(y),
\end{eqnarray*}
and
\begin{eqnarray*}
\partial_{y}\partial_{\mu}f(P_{\mathfrak{v}_0},y)
=g'(\mathbb{E}[h(\mathfrak{v}_0)])(\partial_y^2h)(y).
\end{eqnarray*}
\end{exmp}

 Let us now consider a complete probability space $(\Omega,\mathcal{F},P)$ on which we define a $d$-dimensional Brownian motion $B=(B^1,\cdots,B^d)=(B_t)_{t\in [0,T]}$, where $T\geq 0$ denotes an arbitrarily fixed time horizon. We make the following assumptions: There is a sub-$\sigma$-field $\mathcal{F}_0\subset\mathcal{F}$ such that\\
\indent i) the Brownian motion $B$ is independent of $\mathcal{F}_0$, and\\
\indent ii) $\mathcal{F}_0$ is ``rich enough", i.e., $\mathcal{P}_2(\mathbb{R}^d)=\{ P_{\mathfrak{v}}, \mathfrak{v}\in L^2(\mathcal{F}_0; \mathbb{R}^d) \}.$\\
By $\mathbb{F}=(\mathcal{F}_t)_{t\in [0,T]}$ we denote the filtration generated by $B$, completed and augmented by $\mathcal{F}_0$.\\
\indent Given deterministic Lipschitz functions $\sigma: \mathbb{R}^d\times\mathcal{P}_2(\mathbb{R}^d)\longrightarrow \mathbb{R}^{d\times d}$ and  $b: \mathbb{R}^d\times\mathcal{P}_2(\mathbb{R}^d)\longrightarrow \mathbb{R}^{d},$ we consider for the initial state $(t,x)\in [0,T]\times \mathbb{R}^d$ and $\xi\in L^2(\mathcal{F}_t;\mathbb{R}^d)$ the stochastic differential equations (SDEs)
\begin{equation}
\begin{split}
  \begin{aligned}
X^{t,\xi}_s=\xi+\int^s_t \sigma(X^{t,\xi}_r,P_{X^{t,\xi}_r})dB_r+\int^s_t \sigma(X^{t,\xi}_r,P_{X^{t,\xi}_r})dr,\ s\in[t,T],
  \end{aligned}
  \end{split}
\end{equation}
and
\begin{equation}
\begin{split}
  \begin{aligned}
  X^{t,x,\xi}_s=x+\int^s_t \sigma(X^{t,x,\xi}_r,P_{X^{t,\xi}_r})dB_r+\int^s_t \sigma(X^{t,x,\xi}_r,P_{X^{t,\xi}_r})dr,\ s\in[t,T].
  \end{aligned}
  \end{split}
\end{equation}
It is well-known that under the assumptions above both SDEs have unique solutions in $\mathcal{S}^2([t,T];\mathbb{R}^d),$ which is the space of $\mathbb{F}$-adapted continuous processes $Y=(Y_s)_{s\in[t,T]}$ with $\mathbb{E}[\sup_{s\in [t,T]}|Y_s|^2]\leq \infty.$

\begin{hypothesis}
The couple of coefficients $(\sigma,b)$ belongs to $C^{1,1}_b(\mathbb{R}^d\times \mathcal{P}_2(\mathbb{R}^d)\longrightarrow \mathbb{R}^{d\times d}\times \mathbb{R}^d),$ i.e., the components $\sigma_{i,j},b_j,\ 1\leq i,j \leq d,$ satisfy the following conditions:\\
i) $\sigma_{i,j}(x,\cdot), b_{j}(x,\cdot)$ belong to $C^{1,1}_b( \mathcal{P}_2(\mathbb{R}^d))$, for all $x\in \mathbb{R}^d$\\
ii) $\sigma_{i,j}(\cdot,\mu), b_{j}(\cdot,\mu)$ belong to $C^{1}_b(\mathbb{R}^d)$, for all $\mu\in \mathcal{P}_2(\mathbb{R}^d)$\\
iii) The derivatives $\partial_x\sigma_{i,j}, \partial_xb_{j}:\mathbb{R}^d\times \mathcal{P}_2(\mathbb{R}^d)\longrightarrow \mathbb{R}^d $, $\partial_{\mu}\sigma_{i,j}, \partial_{\mu}b_{j}:\mathbb{R}^d\times \mathcal{P}_2(\mathbb{R}^d)\times \mathbb{R}^d\longrightarrow \mathbb{R}^d $,  are bounded and Lipschitz continuous.

\end{hypothesis}

\begin{hypothesis}\label{H0}
The couple of coefficient $(\sigma,b)$ belongs to $C^{2,1}_b(\mathbb{R}^d\times \mathcal{P}_2(\mathbb{R}^d)\longrightarrow \mathbb{R}^{d\times d}\times \mathbb{R}^d),$ i.e., $(\sigma,b)\in C^{1,1}_b(\mathbb{R}^d\times \mathcal{P}_2(\mathbb{R}^d)\longrightarrow \mathbb{R}^{d\times d}\times \mathbb{R}^d)$ and the components $\sigma_{i,j},b_j,\ 1\leq i,j \leq d,$ satisfies the following conditions:\\
i) $\partial_{x_k}\sigma_{i,j}(\cdot,\cdot), \partial_{x_k} b_{j}(\cdot,\cdot)$ belong to $C^{1,1}_b( \mathbb{R}^d \times\mathcal{P}_2(\mathbb{R}^d))$, for all $1\leq k \leq d;$\\
ii) $\partial_{\mu }\sigma_{i,j}(\cdot,\cdot,\cdot), \partial_{\mu }b_{j}(\cdot, \cdot,\cdot)$ belong to $C^{1,1}_b(\mathbb{R}^d \times \mathcal{P}_2(\mathbb{R}^d)\times \mathbb{R}^d)$, for all $\mu\in \mathcal{P}_2(\mathbb{R}^d)$\\
iii) All the derivatives of $\sigma_{i,j}, b_{j}$, up to order 2 are bounded and Lipschitz continuous.
\end{hypothesis}

\indent The following theorem is taken from \cite{BLP2014}. It gives the It\^{o} formula related to a probability measure.

\begin{theorem}
Let $\Phi \in C^{2,1}_b(\mathbb{R}^d\times \mathcal{P}_2(\mathbb{R}^d)).$ Then, under Hypothesis \ref{H0}, for all $0\leq t \leq s \leq T, x\in \mathbb{R}^d, \xi\in L^2(\mathcal{F}_t;\mathbb{R}^d)$ the It\^o formula is satisfied as follow:
\begin{eqnarray}
&&\Phi(X^{t,x,P_{\xi}}_s, P_{X^{t,\xi}_s})-\Phi(x,P_{\xi})\nonumber\\
&=&\int^s_t\bigg(\sum^d_{i=1}\partial_{x_i}\Phi(X^{t,x,P_{\xi}}_r, P_{X^{t,\xi}_r})b_i(X^{t,x,P_{\xi}}_r, P_{X^{t,\xi}_r})\nonumber\\
&&+\frac{1}{2} \sum^d_{i,j,k=1} \partial^2_{x_i,x_j}\Phi(X^{t,x,P_{\xi}}_r, P_{X^{t,\xi}_r})(\sigma_{i,k}\sigma_{j,k})(X^{t,x,P_{\xi}}_r, P_{X^{t,\xi}_r})\nonumber\\
&&+\mathbb{E}'\big[\sum^d_{i=1}(\partial_{\mu }\Phi)_i(X^{t,x,P_{\xi}}_r, P_{X^{t,\xi}_r},(X^{t,{\xi}}_r)')b_i((X^{t,{\xi}}_r)', P_{X^{t,\xi}_r})\nonumber\\
&&+\frac{1}{2} \sum^d_{i,j,k=1} \partial_{y_i}((\partial_{\mu }\Phi)_j(X^{t,x,P_{\xi}}_r, P_{X^{t,\xi}_r},(X^{t,{\xi}}_r)')(\sigma_{i,k}\sigma_{j,k})((X^{t,{\xi}}_r)', P_{X^{t,\xi}_r})\big]\bigg)dr\nonumber\\
&&+\int^s_t \sum^d_{i,j=1}\partial_{x_i}\Phi(X^{t,x,P_{\xi}}_r, P_{X^{t,\xi}_r})\sigma_{i,j}(X^{t,x,P_{\xi}}_r, P_{X^{t,\xi}_r})dB^j_r, \ s\in [t,T].
\end{eqnarray}

\end{theorem}
\indent For simplicity, we will make use of the following notations concerning matrices. We denote by $\mathbb{R}^{n\times d}$ the space of real matrices of $n\times d $-type, and by $\mathbb{R}^{n\times n}_d$ the linear space of the vectors of matrices $M=(M_1,\cdots, M_d)$, with $M_i\in \mathbb{R}^{n\times n}$, $1\leq i \leq d.$ Given any $\alpha, \beta\in \mathbb{R}^{n}$,  $L, S \in \mathbb{R}^{n\times d}$,  $\gamma \in \mathbb{R}^{d}$ and  $M, N \in \mathbb{R}^{n\times n}_d$, we introduce the following notation: $\alpha \beta=\sum^{n}_{i=1}\alpha_i \beta_i \in \mathbb{R}$, $\alpha\times\beta=(\alpha_i \beta_j)_{1\leq i,j\leq n}$; $LS=\sum^{d}_{i=1} L_i S_i \in \mathbb{R}$, where $L=(L_1,\cdots, L_d), S=(S_1,\cdots, S _d)$;
$ML=\sum^{d}_{i=1}M_i L_i \in \mathbb{R}^{n}$; $M\alpha \gamma=\sum^{d}_{i=1}(M_i \alpha) \gamma_i \in \mathbb{R}^n$; $MN=\sum^{d}_{i=1}M_i N_i \in \mathbb{R}^{n\times n}$;

\indent For mean-field type SDE and BSDE, we have still to introduce some notations. Let $(\tilde{\Omega},\tilde{\mathcal{F}},\tilde{P})$, $(\bar{\Omega},\bar{\mathcal{F}},\bar{P})$ be two copies of the probability space $(\Omega,{\mathcal{F}},P)$. For any random variable $\xi$ over $(\Omega,{\mathcal{F}},P)$, we denote by $\tilde{\xi}$ and $\bar{\xi}$ its copies on $\tilde{\Omega}$ and $\bar{\Omega}$, respectively, which means that they have the same law as $\xi$, but defined over  $(\tilde{\Omega},\tilde{\mathcal{F}},\tilde{P})$ and $(\bar{\Omega},\bar{\mathcal{F}},\bar{P})$. $\mathbb{\tilde{E}}[\cdot]=\int_{\tilde{\Omega}}(\cdot)d\tilde{P}$ and $\mathbb{\bar{E}}[\cdot]=\int_{\bar{\Omega}}(\cdot)d\bar{P}$ act only over the variables from  $\tilde{\omega}$ and $\bar{\omega}$, respectively.

\section{Formulation of the singular optimal stochastic control problem and the main result}

\indent In this section, we  formulate  our generalized mean-field optimal control problem and state the main result of this article. Let $(\Omega,{\mathcal{F}},P)$ be a probability space with filtration ${\mathcal{F}}_{t}$.
Suppose that ${B}_t$
 is a Brownian motion on
$(\Omega,{\mathcal{F}},P)$, where ${\mathcal{F}}$ is the filtration generated by ${B}_t$,  augmented by all $P$-null sets. Let $\mathcal{U}$ denote the admissible control set consisting of $\mathcal{F}_t$-adapted process $u_t$, take values in $U$, such that $\sup_{0\leq t\leq T}\mathbb{E}|u_t|^8 < \infty$, where $U$ is a subset of $\mathbb{R}^k$. Let $b:[0,T]\times \mathbb{R}^n\times \ \mathcal{P}_2(\mathbb{R}^n) \times U \longrightarrow \mathbb{R}^n,\ \sigma:[0,T]\times \mathbb{R}^n\times \ \mathcal{P}_2(\mathbb{R}^n) \longrightarrow \mathbb{R}^{n\times d}, \ h:[0,T]\times \mathbb{R}^n\times \ \mathcal{P}_2(\mathbb{R}^n) \times U \longrightarrow \mathbb{R}$, and $\Phi:\mathbb{R}^n\times \mathcal{P}_2(\mathbb{R}^n) \longrightarrow \mathbb{R}.$

 The state equation and the cost functional are defined by (\ref{eq210c}) and (\ref{eq211c}). Throughout this paper, we make the following assumptions on the coefficients:

\begin{hypothesis}\label{H1}
(1) The functions $b, \sigma, h ,\Phi$ are differentiable with respect to $(x,\mu, v).$ $b,\sigma$ satisfy Lipschitz condition with respect to $(x,\mu,v).$\\
(2) The first-order derivatives with respect to $(x,\mu)$ of $b, \sigma$ are Lipschitz continuous and bounded.\\
(3)  The first-order derivatives  with respect to $(x,\mu)$ of $h, \Phi$ are Lipschitz continuous and bounded by $C(1+|x|+|v| )$.\\
(4) The second-order derivatives with respect to $(x,\mu) $ of $b, \sigma, h ,\Phi$ are continuous and bounded. All the second-order derivatives are Borel measurable with respect to $(t,x,\mu,v).$
\end{hypothesis}

\indent Suppose that $u$ is an optimal control and $X^u$ is the associated trajectory. We are to find the necessary conditions satisfied by $u$. Firstly, we introduce the following abbreviations:
 $$b(t):=b(t,X^u_t, P_{X^u_t},u_t),
 \ b_x(t):=b_x(t,X^u_t,P_{X^u_t},u_t),$$
 $$\tilde{b}(t):=b(t,\tilde{X}^u_t, P_{\tilde{X}^u_t},\tilde{u}_t), \ \tilde{b}_x(t):=b_x(t,\tilde{X}^u_t,P_{\tilde{X}^u_t},\tilde{u}_t),$$
 $$b_{xx}(t):=b_{xx}(t,X^u_t,P_{X^u_t},u_t),\ b_{\mu }(t):=b_{\mu }(t,X^u_t,P_{X^u_t},\tilde{X}^u_t, u_t),$$
 $$ \tilde{b}_{\mu }(t):=b_{\mu }(t,\tilde{X}^u_t,P_{\tilde{X}^u_t},{X}^u_t, \tilde{u}_t), b_{\mu \mu }(t):=b_{\mu \mu}(t,X^u_t,P_{X^u_t},\tilde{X}^u_t, \bar{X}^u_t, u_t),$$
 $$\tilde{\bar{b}}_{\mu \mu }(t):=b_{\mu \mu}(t,\tilde{X}^u_t,P_{\tilde{X}^u_t},\bar{X}^u_t, {X}^u_t, \tilde{u}_t), \ b_{x  \mu  }(t):=b_{x \mu }(t,X^u_t,P_{X^u_t},\tilde{X}^u_t, u_t),$$
 $$\ b_{y  \mu  }(t):=b_{y \mu }(t,X^u_t,P_{X^u_t},\tilde{X}^u_t, u_t), \ \bigtriangleup b(t;v):=b(t,X^u_t,P_{X^u_t},v)-b(t),$$
 $$\bigtriangleup b_x(t;v):=b_x(t,X^u_t,P_{X^u_t},v)-b_x(t), \ \bigtriangleup b_{\mu }(t;v):=b_{\mu }(t,X^u_t,P_{X^u_t},\tilde{X}^u_t, v)-b_{\mu }(t),$$

 Similar shorthand notations for the second-order derivatives and those about $ \sigma, h $ can also be introduced.\\
\indent Consider the first order adjont process
\begin{equation}\label{eq210d}
\left\{
\begin{array}{rcl}
 -dp_t&=& \bigg\{ b_x(t)p_t+\sigma_x(t)q_t+h_x(t)\\
&&+\tilde{\mathbb{E}}\Big[\tilde{b}_{\mu }(t)\tilde{p}_t+\tilde{\sigma}_{\mu }(t)\tilde{q}_t+\tilde{h}_{\mu }(t)\Big] \bigg\} dt-q_tdB_t,\\
p_T&=&\Phi_x(X^u_T,P_{X^u_T})+\tilde{\mathbb{E}}\left[\Phi_{\mu }(\tilde{X}^u_T,P_{\tilde{X}^u_T},X^u_T)\right].
   \end{array}
  \right.
\end{equation}
According to Theorem 3.1 \cite{BLP2009}, this BSDE admit a unique adapted solution. We also denote the solution as $(p^u_t, q^u_t).$ Define the Hamiltonian as follows:
$$H(t,x,\mu ,p,q,v)=pb(t,x,\mu,v)+q\sigma(t,x,\mu)+h(t,x,\mu,v)$$

The following first-order SMP is obtained as a special case of \cite{BLM2016}.

\begin{theorem}[The First Order SMP] Let Hypothesis \ref{H1} hold. Suppose that $X^u_t$ is the associated trajectory of the optimal control $u$, and $(p,q)$ is the solution to the mean-field backward stochastic differential equation (MFBSDE) (\ref{eq210d}). Then, there is a subset $I_0 \subset [0,T]$ which is of full measure such that  $\forall t\in I_0$,
\begin{eqnarray}
  H(t,X^u_t,P_{X^u_t},p_t,q_t,u_t)=\inf_{v\in \mathcal{U}}H(t,X^u_t,P_{X^u_t},p_t,q_t,v), \ \ a.s..
  \end{eqnarray}
\end{theorem}

 As we pointed out in the introduction, the aim of this article is to derive another SMP when the Hamiltonian function  above becomes singular, and hence, the SMP above is not suitable for characterizing of the optimal control $u_t$. To this end, we define the second-order adjoint process as follows:
\begin{equation}\label{eq210b}
\left \{
\begin{array}{rcl}
dP_t&=&-\bigg \{ {b}^*_x(t)P_t+P_tb_x(t)+\tilde{\mathbb{E}}\left[\tilde{b}_{\mu}^*(t)\right]{P}_t +{P}_t\tilde{\mathbb{E}}\left[\tilde{b}_{\mu}(t) \right]\\
&&+{\sigma}^*_x(t)P_t\sigma_x(t)
+\tilde{\mathbb{E}}\left[\tilde{\sigma}^*_{\mu}(t)\right]P_t\tilde{\mathbb{E}}\left[\tilde{\sigma}_{\mu}(t) \right]\\
&&+{\sigma}^*_{x}(t)P_t\tilde{\mathbb{E}}\left[\tilde{\sigma}_{\mu}(t) \right]+\tilde{\mathbb{E}}\left[\tilde{\sigma}^*_{\mu}(t)\right]P_t {\sigma}_{x}(t) \\
&&+{\sigma}^*_x(t)Q_t+P_t\sigma_x(t)+\tilde{\mathbb{E}}\left[\tilde{\sigma}_{\mu}^*(t)\right]{Q}_t +{Q}_t\tilde{\mathbb{E}}\left[\tilde{\sigma}_{\mu}(t) \right]\\
&&+H_{xx}(t)+\tilde{\mathbb{E}}\bar{\mathbb{E}}\left[\tilde{\bar{H}}_{\mu\mu}(t) \right]+\tilde{\mathbb{E}}\left[\tilde{H}_{y\mu}(t) \right]+2\tilde{\mathbb{E}}\left[\tilde{H}_{x\mu}(t) \right]\bigg \}dt\\
&&+Q_tdB_t,\\
P_T&=&0.
\end{array}
\right.
\end{equation}

\begin{remark}
By changing the terminal condition $p_T$, we can always eliminate the terminal cost when deducing the variational inequality. In fact, the terminal condition $P_T=0$ is due to the assumption that $\Phi\equiv 0$. Without this assumption, we only need to set
\begin{eqnarray}
P_T&=&\Phi_{xx}(X^u_T,P_{X^u_T})+2\tilde{\mathbb{E}}\left[\Phi_{x\mu}(X^u_T,P_{X^u_T},\tilde{X}^u_T)\right] \nonumber\\
&&+\bar{\mathbb{E}}\tilde{\mathbb{E}}\left[\Phi_{\mu\mu}(X^u_T,P_{X^u_T},\tilde{X}^u_T,\bar{X}^u_T)\right]+\tilde{\mathbb{E}}\left[\Phi_{y\mu}(X^u_T,P_{X^u_T},\tilde{X}^u_T)\right].
\end{eqnarray}
Without loss of generality, we assume the terminal cost $\Phi \equiv 0$ in the following sections.
\end{remark}

Finally, we present our main result in this article.
\begin{theorem}\label{thm0209a}
Assume that Hypothesis \ref{H1} hold. Let $\big(X^u_{\cdot}, u_{\cdot}\big)$ be an optimal pair and let $u_{\cdot}$ be singular on the control region $V$. Suppose that $(P, Q)$ is the unique adapted solution of equation (\ref{eq210b}). Then, there is a full measure subset $I_0\subset [0,T]$ such that at each $t\in I_0$, $\big(X^u_{\cdot}, u_{\cdot}\big)$ satisfies, not only the first-order stochastic maximum principle, but also the following inequality
 \begin{eqnarray}
 &&\bigtriangleup H_x(t;v)\bigtriangleup b(t;v)+\tilde{\mathbb{E}}\left[\bigtriangleup H_\mu(t;v)\bigtriangleup \tilde{b}(t;v)\right]\nonumber\\
&&+\bigtriangleup {b}^*(t;v) P_t\bigtriangleup b(t;v)\geq 0,\ \forall v\in U, a.s..
\end{eqnarray}
\end{theorem}

\section{Quantitative analysis of the impact of control actions on the state}
\indent In this section, we expand the state process according to different orders of the perturbation parameter $d(u,v)$, a distance between the optimal control $u$ and its perturbation $v$.

\begin{lemma}
Under Hypothesis \ref{H1} on the coefficients, we have,
$$\mathbb{E} \sup_{0\leq t \leq T}|X^{v}_t|^8 \leq K\left(1+\mathbb{E}\left|\int^T_0 |v_s|ds\right|^8 \right), $$
\end{lemma}

\begin{proof}
By the state equation (\ref{eq210c}), for $\tau\in[0,T]$ we have,
\begin{eqnarray}
  \mathbb{E}\sup_{0\leq t \leq \tau}|X^v_t|^8 &\leq &K \mathbb{E}\bigg(|x|^8 +\sup_{0\leq t \leq \tau}\left|\int^t_0 b(s,X^v_s,P_ {{X^v_s}},v_s) ds\right|^8\nonumber \\
  &&+\Big(\int^\tau_0 |\sigma(s,X^v_s,P_ {{X^v_s}})|^2 ds\Big)^4\bigg)\nonumber \\
  &\leq &K\big(|x|^8+ \mathbb{E}\int^\tau_0\sup_{0\leq s \leq r}| X^v_s|^8 dr+ \mathbb{E}\left|\int^T_0 |v_s| ds\right|^8\bigg)
\end{eqnarray}
From Gronwall's inequality, we then have the desired result.
\end{proof}

\indent For $v_i \in \mathcal{U}$, $i=1,2$, we define
$$I(v_1,v_2)=\left\{ t\in [0,T]\big |P(\{ \omega : v_1(t)\neq v_2(t)\} )>0 \right\}$$
and $d(v_1,v_2)=|I(v_1,v_2)|$ is the Lebesgue measure of $I(v_1,v_2)$. Then, $(\mathcal{U},d)$ is a metric space.
 
 Given the optimal pair $(X^u_\cdot, u_\cdot)$, we now proceed to the perturbation $X^v$ of $X^u$.\\
 Let
 \begin{eqnarray}\label{eq41}
  X^{v,1}_t&=&\int^t_0 \bigg\{ b_x(s)X^{v,1}_s+\tilde{\mathbb{E}}\left[b_{\mu}(s) \tilde{X}^{v,1}_s\right]+\bigtriangleup b(s,v) \bigg\}ds\nonumber\\
  &&+\int^t_0\bigg \{ \sigma_x(s)X^{v,1}_s+\tilde{\mathbb{E}}\left[\sigma_{\mu}(s) \tilde{X}^{v,1}_s\right]\bigg \}dB_s\nonumber\\
  &:=&\int^t_0  b^1(s,v) ds+\int^t_0 \sigma^1(s,v)dB_s,
\end{eqnarray}
and
 \begin{eqnarray}\label{eq42}
  X^{v,2}_t&=&\int^t_0\bigg \{ b_x(s)X^{v,2}_s+\tilde{\mathbb{E}}\left[b_{\mu}(s) \tilde{X}^{v,2}_s\right]+\bigtriangleup b_x(s,v)X^{v,1}_s+\tilde{\mathbb{E}}\left[\bigtriangleup b_{\mu}(s,v)\tilde{X}^{v,1}_s\right] \nonumber\\
 &&+\frac{1}{2} b_{xx}(s)X^{v,1}_s\times X^{v,1}_s+ \tilde{\mathbb{E}}\left[ b_{x \mu}(s)X^{v,1}_s\times \tilde{X}^{v,1}_s\right]\nonumber\\
 &&+\frac{1}{2} \tilde{\mathbb{E}}\left[ b_{y \mu}(s)\tilde{X}^{v,1}_s\times \tilde{X}^{v,1}_s\right]+\frac{1}{2}\bar{\mathbb{E}} \tilde{\mathbb{E}}\left[ b_{\mu \mu}(s)\tilde{X}^{v,1}_s\times \bar{X}^{v,1}_s\right] \bigg\}ds\nonumber\\
 &&+\int^t_0 \bigg\{ \sigma_x(s)X^{v,2}_s+\tilde{\mathbb{E}}\left[\sigma_{\mu}(s) \tilde{X}^{v,2}_s\right]\nonumber\\
 &&+\frac{1}{2} \sigma_{xx}(s)X^{v,1}_s\times X^{v,1}_s+\tilde{\mathbb{E}}\left[ \sigma_{x \mu}(s)X^{v,1}_s\times \tilde{X}^{v,1}_s\right]\nonumber\\
 &&+\frac{1}{2}\tilde{\mathbb{E}}\left[ \sigma_{y \mu}(s)\tilde{X}^{v,1}_s\times \tilde{X}^{v,1}_s\right]+\frac{1}{2}\bar{\mathbb{E}} \tilde{\mathbb{E}}\left[ \sigma_{\mu \mu}(s)\tilde{X}^{v,1}_s\times \bar{X}^{v,1}_s\right] \bigg\}dB_s\nonumber\\
   &:=&\int^t_0  b^2(s,v) ds+\int^t_0 \sigma^2(s,v)dB_s
  \end{eqnarray}
 Denote
  \begin{eqnarray} \label{eq43}
  X^{v*}_{.}:=X^*_{.}(v)=X^u_{.}-X^{v,1}_{.}-X^{v,2}_{.}-X^v_{.}.
 \end{eqnarray}
  The following lemmas give the estimation of their orders according to parameter $d(v,u)$.
\begin{lemma}\label{lm42}
Assume that Hypothesis \ref{H1} holds. Then, there exists a $K> 0$, such that for any $v(\cdot), u(\cdot)\in \mathcal{U}$, we have 
\begin{eqnarray}
 \mathbb{E}\sup_{0\leq t \leq T}|X^{v,1}_t|^2&\leq&Kd^2(v,u),\\\label{eq411}
 \mathbb{E}\sup_{0\leq t \leq T}|X^{v,2}_t|^2&\leq&Kd^{4}(v,u)\label{eq412}.
 \end{eqnarray}
\end{lemma}

\begin{proof}
For any $\tau\in[0,T]$, denote 
\begin{eqnarray*}
g_1(\tau)=\mathbb{E}\sup_{0\leq t \leq \tau}|X^{v,1}_t|^2, \ g_2(\tau)=\mathbb{E}\sup_{0\leq t \leq \tau}|X^{v,2}_t|^2.
\end{eqnarray*}

By Hypothesis 3.1 and the Burkholder-Davis-Gundy inequality, we have
\begin{eqnarray}
g_1(\tau)\leq K\left(\int^\tau_0 g_1(s) ds+\mathbb{E}\left|\int^T_0|\bigtriangleup b(s;v)|ds\right|^2\right),
\end{eqnarray}
and
\begin{eqnarray}
g_2(\tau)&\leq &K\bigg(\int^\tau_0 g_2(s) ds+[g_1(T)]^2\nonumber\\
&&+\mathbb{E}\left|\int^T_0|\bigtriangleup b_x(s;v)|ds\right|^{4}+\mathbb{E}\tilde{\mathbb{E}}\left|\int^T_0|\bigtriangleup b_\mu(s;v)|ds\right|^{4}\bigg).
 \end{eqnarray}
The application of Grownwall's inequality allows to obtain that
\begin{eqnarray}\label{eq410}
 \mathbb{E}\sup_{0\leq t \leq T}|X^{v,1}_t|^2&\leq& K\left(\mathbb{E}\left|\int^T_0|\bigtriangleup b(s;v)|ds\right|^2\right),\\
 \mathbb{E}\sup_{0\leq t \leq T}|X^{v,2}_t|^2&\leq& K\Bigg(\mathbb{E}\bigg|\int^T_0|\bigtriangleup b(s;v)|ds\bigg|^{4}+\mathbb{E}\bigg|\int^T_0 |\bigtriangleup b_x(s;v)|ds\bigg|^{4}\nonumber\\
&&+\mathbb{E}\tilde{\mathbb{E}}\bigg|\int^T_0|\bigtriangleup b_{\mu}(s;v)|ds \bigg|^{4}\Bigg)\label{eq45}
 \end{eqnarray}
\indent Notice that the first-order derivative $b_x$ is bounded. Then, (\ref{eq45}) implies the following estimate
\begin{eqnarray}
 \mathbb{E}\sup_{0\leq t\leq T}|X^{v,2}_t|^2 \leq K\left(\mathbb{E}\left|\int^T_0|\bigtriangleup b(s;v)|ds\right|^{4}+d^{4}(u,v)\right).
 \end{eqnarray}
According to assumption about $v$ and $u$, then,
\begin{eqnarray}
 \mathbb{E}\sup_{0\leq t\leq T}|X^{v,1}_t|^2 \leq Kd^{2}(v, u),\mbox{ and }\ 
  \mathbb{E}\sup_{0\leq t\leq T}|X^{v,2}_t|^2 \leq Kd^{4}(v, u).
 \end{eqnarray}

In fact, by Minkowski's inequality, we have
\begin{eqnarray}
 &&\mathbb{E}\Bigg|\int^T_0|\bigtriangleup b(s;v)|ds\Bigg|^{4}\leq\Bigg|\int^T_0 \bigg(\mathbb{E}|\bigtriangleup b(s;v)|^{4}\bigg)^{\frac{1}{4}}ds\Bigg|^{4}\nonumber \\
 &=&\Bigg|\int_{I(u,v)} \bigg(\mathbb{E}|\bigtriangleup b(s;v)|^{4}\bigg)^{\frac{1}{4}}ds\Bigg|^{4}\le Kd^{4}(u,v). 
 \end{eqnarray}
\end{proof}

 The following lemma gives the order of ${X^{v*}_t}$.
\begin{lemma}\label{lm43}
Assume Hypothesis \ref{H1} holds.
For $v(\cdot)\in \mathcal{U}$ and Borel subset $I_{\rho}\subset [0,T]$ with Lebesgue measure $|I_{\rho}|$, define
 \begin{eqnarray}\label{eq49}
\hat{v}_t=v_t\text{1}_{I_{\rho}}(t)+u_t\text{1}_{[0,T]\setminus I_{\rho}}(t),\ X^{\hat{v}*}_t:=X^{*}(t,\hat{v}).
 \end{eqnarray}
Then we have
 \begin{eqnarray}
&\mathbb{E}\sup \limits_{0\leq t \leq T}|X^{\hat{v}*}_t|^2=o(|I_{\rho}|^{4})&
   \end{eqnarray}
   when  $|I_{\rho}|\rightarrow 0$.
\end{lemma}

\begin{proof}

We introduce the following notations first
$$\bigtriangleup b_{xx}(s;\lambda\eta;v):=b_{xx}(s,X^u_s+\lambda \eta X^{v,12}_s, P_{X^u_s+\lambda \eta X^{v,12}_s},v_s)-b_{xx}(s)$$
$$\bigtriangleup b_{\mu\mu}(s;\lambda\eta;v):=b_{\mu\mu}(s,X^u_s, P_{X^u_s+\lambda \eta X^{v,12}_s},\tilde{X}^u_s+\lambda \eta \tilde{X}^{v,12}_s,\bar{X}^u_s+\lambda \eta  \bar{X}^{v,12}_s,v_s)-b_{xx}(s),$$
$$\bigtriangleup b_{\mu y}(s;\lambda\eta;v)=b_{\mu y}(s,X^u_s ,P_{X^u_s+\lambda \eta X^{v,12}_s},\tilde{X}^u_s+\lambda \eta  \tilde{X}^{v,12}_s,v_s)-b_{\mu y}(s)$$
$$\bigtriangleup b_{x\mu}(s;\lambda;v):=b_{x\mu}(s,X^u_s, P_{X^u_s+\lambda X^{v,12}_s},\tilde{X}^u_s+\lambda   \tilde{X}^{v,12}_s,v_s)-b_{x\mu}(s),$$
where $ X^{v,12}_{\cdot}:=X^{v,1}_{\cdot}+X^{v,2}_{\cdot}.$
Similarly notations can be introduced with $b$ replaced by $\sigma.$\\
\indent We now proceed to estimating $X^*_{.}(\hat{v})$ defined by (\ref{eq43}). By (\ref{eq210c}), (\ref{eq41}) and (\ref{eq42}), we have
\begin{eqnarray}
dX^{\hat{v}*}_t=\alpha(t)dt+\beta(t)dB_t,
\end{eqnarray}
where
$$\alpha(t)=b(t,X^{\hat{v}},P_{X^{\hat{v}}},\hat{v}_t)-\big[b(t) +b^1(t,\hat{v})+b^2(t,\hat{v})\big],$$
 $$\beta(t)=\sigma(t,X^{\hat{v}}_t,P_{X^{\hat{v}}_t})-\big[\sigma(t) +\sigma^1(t,{\hat{v}})+\sigma^2(t,{\hat{v}})\big].$$
\indent We can represent $\alpha(t)$ as follows. 
\begin{eqnarray}\label{eq44}
\alpha(t)&=&b(t,X^{\hat{v}}_t,P_{X^{\hat{v}}_t},{\hat{v}_t})-\bigg[b(t) +\Big\{ b_x(t)X^{{\hat{v}},1}_t+\tilde{\mathbb{E}}\left[b_{\mu}(t) \tilde{X}^{{\hat{v}},1}_t\right]+\bigtriangleup b(t,{\hat{v}})\Big \}\nonumber\\
&&+\Big\{ b_x(t)X^{{\hat{v}},2}_t+\tilde{\mathbb{E}}\left[b_{\mu}(t) \tilde{X}^{{\hat{v}},2}_t\right]+\bigtriangleup b_x(t,{\hat{v}})X^{{\hat{v}},1}_t+\tilde{\mathbb{E}}\left[\bigtriangleup b_{\mu}(t,{\hat{v}})\tilde{X}^{{\hat{v}},1}_t\right] \nonumber\\
 &&+\frac{1}{2} b_{xx}(t)X^{{\hat{v}},1}_t\times X^{{\hat{v}},1}_t+ \tilde{\mathbb{E}}\left[ b_{x \mu}(t)X^{{\hat{v}},1}_t\times \tilde{X}^{{\hat{v}},1}_t\right]\nonumber\\
 &&+\frac{1}{2} \tilde{\mathbb{E}}\left[ b_{y \mu}(t)\tilde{X}^{{\hat{v}},1}_t\times \tilde{X}^{{\hat{v}},1}_t\right]+\frac{1}{2}\bar{\mathbb{E}} \tilde{\mathbb{E}}\left[ b_{\mu \mu}(t)\tilde{X}^{{\hat{v}},1}_t\times \bar{X}^{{\hat{v}},1}_t\right] \Big\}\bigg].\nonumber\\
\end{eqnarray}
   Denote
\begin{eqnarray}\label{eq46}
A(t;{\hat{v}})&=&\frac{1}{2} b_{xx}(t)(X^{{\hat{v}},2}_t\times X^{{\hat{v}},2}_t+2X^{{\hat{v}},1}_t \times X^{{\hat{v}},2}_t)\nonumber\\
&&+\frac{1}{2}\tilde{\mathbb{E}}\left[ b_{\mu y}(t)(\tilde{X}^{{\hat{v}},2}_t\times \tilde{X}^{{\hat{v}},2}_t+2\tilde{X}^{{\hat{v}},1}_t \times \tilde{X}^{{\hat{v}},2}_t)\right]\nonumber\\
&&+\frac{1}{2}\tilde{\mathbb{E}}\bar{\mathbb{E}}\left[ b_{\mu \mu}(t)(\tilde{X}^{{\hat{v}},2}_t\times \bar{X}^{{\hat{v}},2}_t+\tilde{X}^{{\hat{v}},1}_t \times \bar{X}^{{\hat{v}},2}_t+\bar{X}^{{\hat{v}},1}_t \times \tilde{X}^{{\hat{v}},2}_t)\right]\nonumber\\
&&+\bigtriangleup b_x (t;{\hat{v}})X^{{\hat{v}},2}_t+\tilde{\mathbb{E}}\left[\bigtriangleup b_\mu (t;{\hat{v}})\tilde{X}^{{\hat{v}},2}_t\right]\nonumber\\
&&+\tilde{\mathbb{E}}\left[b_{x\mu}(t)(X^{{\hat{v}},2}_t\times \tilde{X}^{{\hat{v}},2}_t+X^{{\hat{v}},1}_t \times \tilde{X}^{{\hat{v}},2}_t+\tilde{X}^{{\hat{v}},1}_t \times X^{{\hat{v}},2}_t)\right]\nonumber\\
  &&+\int^1_0 \int^1_0 \lambda \left[\bigtriangleup b_{xx}(t;\lambda\eta;{\hat{v}})d \lambda d\eta X^{{\hat{v}},12}_t \times X^{{\hat{v}},12}_t \right]\nonumber\\
 &&+\int^1_0 \int^1_0\bar{\mathbb{E}}\tilde{\mathbb{E}} \left[\lambda (\bigtriangleup b_{\mu\mu}(t;\lambda\eta;{\hat{v}}) d \lambda d\eta  \tilde{X}^{{\hat{v}},12}_t \times \bar{X}^{{\hat{v}},12}_t\right]\nonumber\\
 &&+\int^1_0 \int^1_0\tilde{\mathbb{E}} \left[\lambda (\bigtriangleup b_{\mu y}(t;\lambda\eta;{\hat{v}})d \lambda d\eta  \tilde{X}^{{\hat{v}},12}_t \times \tilde{X}^{{\hat{v}},12}_t\right]\nonumber\\
 &&+\tilde{\mathbb{E}}\left[\int^1_0 \bigtriangleup b_{x\mu}(t;\lambda;{\hat{v}})d\lambda X^{{\hat{v}},12}_t\times \tilde{X}^{{\hat{v}},12}_t\right].
    \end{eqnarray}
It is easy to show that
\begin{eqnarray*}
\alpha(t)&=& b(t,X^{\hat{v}}_t,P_{X^{\hat{v}}_t},{\hat{v}}_t)-\Big[b(t,X^u_t,P_{X^u_t+X^{{\hat{v}},12}_t},{\hat{v}}_t)\nonumber\\
&&+ b_x(t,X^u_t,P_{X^u_t+X^{{\hat{v}},12}_t},{\hat{v}}_t)X^{{\hat{v}},12}_t \nonumber\\
 &&+\int^1_0 \int^1_0 \lambda b_{xx}(t,X^u_t+\lambda \eta X^{{\hat{v}},12}_t,P_{X^u_t+X^{{\hat{v}},12}_t},{\hat{v}}_t)d \lambda d\eta X^{{\hat{v}},12}_t\times X^{{\hat{v}},12}_t\Big]\nonumber\\
 &&+A(t;{\hat{v}})\\
&=& b(t,X^{\hat{v}}_t,P_{X^{\hat{v}}_t},{\hat{v}}_t)-b(t,X^u_t+X^{{\hat{v}},12}_t,P_{X^u_t+X^{{\hat{v}},12}_t},{\hat{v}}_t)+A(t;{\hat{v}}).
\end{eqnarray*}
 Simularly, by setting 
    \begin{eqnarray}\label{eq47}
B(t;{\hat{v}})&:=&\frac{1}{2} \sigma_{xx}(t)(X^{{\hat{v}},2}_t\times X^{{\hat{v}},2}_t+2X^{{\hat{v}},1}_t \times X^{{\hat{v}},2}_t)\nonumber\\
&&+\frac{1}{2}\tilde{\mathbb{E}}\left[ \sigma_{\mu y}(t)(\tilde{X}^{{\hat{v}},2}_t\times \tilde{X}^{{\hat{v}},2}_t+2\tilde{X}^{{\hat{v}},1}_t \times \tilde{X}^{{\hat{v}},2}_t)\right]\nonumber\\
&&+\frac{1}{2}\tilde{\mathbb{E}}\bar{\mathbb{E}}\left[ \sigma_{\mu \mu}(t)(\tilde{X}^{{\hat{v}},2}_t\times \bar{X}^{{\hat{v}},2}_t+\tilde{X}^{{\hat{v}},1}_t\times \bar{X}^{{\hat{v}},2}_t+\bar{X}^{{\hat{v}},1}_t \times \tilde{X}^{{\hat{v}},2}_t)\right]\nonumber\\
&&+\tilde{\mathbb{E}}\left[\sigma_{x\mu}(t)(X^{{\hat{v}},2}_t\times \tilde{X}^{{\hat{v}},2}_t+X^{{\hat{v}},1}_t \times \tilde{X}^{{\hat{v}},2}_t+\tilde{X}^{{\hat{v}},1}_t \times X^{{\hat{v}},2}_t)\right]\nonumber\\
  &&+\int^1_0 \int^1_0 \lambda \left[\bigtriangleup \sigma_{xx}(t;\lambda\eta;{\hat{v}})d \lambda d\eta X^{{\hat{v}},12}_t \times X^{{\hat{v}},12}_t \right]\nonumber\\
 &&+\int^1_0 \int^1_0\bar{\mathbb{E}}\tilde{\mathbb{E}} \left[\lambda \bigtriangleup \sigma_{\mu\mu}(t;\lambda\eta;{\hat{v}}) d \lambda d\eta  \tilde{X}^{{\hat{v}},12}_t \times  \bar{X}^{{\hat{v}},12}_t\right]\nonumber\\
 &&+\int^1_0 \int^1_0\tilde{\mathbb{E}} \left[\lambda \bigtriangleup \sigma_{\mu y}(t;\lambda\eta;{\hat{v}})d \lambda d\eta \tilde{X}^{{\hat{v}},12}_t \times \tilde{X}^{{\hat{v}},12}_t\right]\nonumber\\
 &&+\tilde{\mathbb{E}}\left[\int^1_0 (\bigtriangleup \sigma_{x\mu}(t;\lambda;{\hat{v}})d\lambda X^{{\hat{v}},12}_t\times \tilde{X}^{{\hat{v}},12}_t)d\lambda X^{{\hat{v}},12}_t\times \tilde{X}^{{\hat{v}},12}_t\right],
      \end{eqnarray}
  we have 
  \begin{eqnarray*}
 \beta(t)&=& \sigma(t,X^{\hat{v}}_t,P_{X^{\hat{v}}_t})-\Big[\sigma(t,X^u_t,P_{X^u_t+X^{{\hat{v}},12}_t})+ \sigma_x(t,X^u_t,P_{X^u_t+X^{{\hat{v}},12}_t})X^{{\hat{v}},12}_t \nonumber\\
 &&+\int^1_0 \int^1_0 \lambda \sigma_{xx}(t,X^u_t+\lambda \eta X^{{\hat{v}},12}_t,P_{X^u_t+X^{{\hat{v}},12}_t})d \lambda d\eta X^{{\hat{v}},12}_t\times X^{{\hat{v}},12}_t\Big]\nonumber\\
 &&+B(t;{\hat{v}})\\
&=& \sigma(t,X^{\hat{v}}_t,P_{X^{\hat{v}}_t})-\sigma(t,X^u_t+X^{{\hat{v}},12}_t,P_{X^u_t+X^{{\hat{v}},12}_t})+B(t;{\hat{v}}).
  \end{eqnarray*}
\indent  According to Hypothesis \ref{H1}, we have
\begin{eqnarray*}
&&  \left|b(t,X^{\hat{v}}_t,P_{X^{\hat{v}}_t},{\hat{v}}_t)-b(t,X^u_t+X^{{\hat{v}},12}_t,P_{X^u_t+X^{{\hat{v}},12}_t},{\hat{v}}_t)\right|\\
  &\leq& K\left(|X^{{\hat{v}}*}_t|+W_2(P_{X^{\hat{v}}_t},P_{X^u_t+X^{{\hat{v}},12}_t})\right).
  \end{eqnarray*}  
 Note that
 \begin{eqnarray*}
 W_2(P_{X^{\hat{v}}_s}, P_{X^u_s+ X^{{\hat{v}},12}_s})^2\leq&\mathbb{E}\left|X^{\hat{v}}_s-X^u_s-X^{{\hat{v}},12}_s\right|^2=\mathbb{E}\left|X^{{\hat{v}}*}_s\right|^2.
  \end{eqnarray*}
   By  Burkholder-Davis-Gundy inequality, for $\tau\in[0,T]$, we obtain the following estimation
\begin{eqnarray}
\mathbb{E}\sup_{0 \leq t \leq \tau}|X^{{\hat{v}}*}_t|^2 &\leq & \int ^\tau_0 K \mathbb{E} \sup_{0\leq r \leq  s} |X^{{\hat{v}}*}_r|^2 ds+\mathbb{E}\int^{T}_0 |A(s;{\hat{v}})|^2ds \nonumber\\
&&+\mathbb{E}\int^{T}_0 |B(s;{\hat{v}})|^2ds.
 \end{eqnarray}
 According to Gronwall's inequality, we have
 \begin{eqnarray}\label{eq48}
\mathbb{E}\sup_{0 \leq t \leq T}|X^{{\hat{v}}*}_t|^2 &\leq & K \bigg(\mathbb{E}\int^{T}_0 |A(s;{\hat{v}})|^2ds
+\mathbb{E}\int^{T}_0 |B(s;{\hat{v}})|^2ds\bigg).
 \end{eqnarray}
About $A(s;{\hat{v}})$ we have 
   \begin{eqnarray}
&& \mathbb{E}\int^T_0|A(s;{\hat{v}})|^{2}ds\nonumber\\
 &\leq &K \mathbb{E}\bigg( \sup_{0\leq s \leq t}|X^{{\hat{v}},2}_s|^{4}+\sup_{0\leq s \leq t}|X^{{\hat{v}},1}_s|^2\sup_{0\leq s \leq t}|X^{{\hat{v}},2}_s|^2\nonumber\\
 &&+\sup_{0\leq s \leq t}|X^{{\hat{v}},2}_s|^2\int^T_0\left[\left|\bigtriangleup b_x(s;{\hat{v}})\right|^2
+\left|\bigtriangleup b_\mu(s;{\hat{v}})\right|^2\right]ds\nonumber\\
 &&+\sup_{0\leq s \leq t}|X^{{\hat{v}},1}_s|^{4}\int^T_0\int^1_0 \int^1_0 |\lambda \bigtriangleup b_{xx}(s;\lambda\eta;{\hat{v}})|^{2} d \lambda d\eta ds \nonumber\\
 &&+\sup_{0\leq s \leq t}|X^{{\hat{v}},1}_s|^{4}\int^T_0\int^1_0 \int^1_0\bar{\mathbb{E}}\tilde{\mathbb{E}} \left|\lambda \bigtriangleup b_{\mu\mu}(s;\lambda\eta;{\hat{v}})\right|^2\cdot d \lambda d\eta ds \nonumber\\
 &&+\sup_{0\leq s \leq t}|X^{{\hat{v}},1}_s|^{4}\int^T_0\int^1_0 \int^1_0\tilde{\mathbb{E}} \left|\lambda \bigtriangleup b_{\mu y}(s;\lambda\eta;{\hat{v}})\right|^2d \lambda d\eta ds\nonumber\\
 &&+\sup_{0\leq s \leq t}|X^{{\hat{v}},1}_s|^{4}\tilde{\mathbb{E}}\int^T_0\int^1_0 \left|\bigtriangleup b_{x\mu}(s;\lambda;{\hat{v}})\right|^2d\lambda ds\bigg).
  \end{eqnarray}
  Note that
  $$\mathbb{E}\left|\int^T_0|\bigtriangleup b(s;{\hat{v}})|ds\right|^{8}\leq K|I_{\rho}|^8,$$
  similar estimates hold with $b$ replaced by $b_{x}$ and $b_{\mu}$. Since
  $X^{{\hat{v}},12}_t\rightarrow 0$ as $|I_{\rho}|\rightarrow 0$, so we also have
  $$ \bigtriangleup b_{xx}(s;\lambda\eta;{\hat{v}})\rightarrow 0, $$
  replace $b_{xx}$ by $b_{\mu\mu}$, $b_{\mu y}$ and $b_{x\mu}$, we can get the similar result when $|I_{\rho}|\rightarrow 0$.\\ 
  \indent According to estimation of $X^{{\hat{v}},1}_{\cdot}, \ X^{{\hat{v}},2}_{\cdot}$ in Lemma \ref{lm42}, we obtain
    \begin{eqnarray}
&&\mathbb{E} \int^T_0 |A(s;{\hat{v}})|^2ds~~~~~~~~~~~~~~~~~~~~~~~~~~~~~~~~~~~~~~~~~~~~~~~~~~~~~~~~~~~~~~~~~~~~~~~~~~~\nonumber\\
& \leq &K\bigg(\sqrt{\mathbb{E}\left|\int^T_0|\bigtriangleup b(s;{\hat{v}})|ds\right|^{8}}+\sqrt{\mathbb{E}\left|\int^T_0|\bigtriangleup b_x(s;{\hat{v}})|ds\right|^{8}}\nonumber\\
 &&+\sqrt{\mathbb{E}\tilde{\mathbb{E}}\left|\int^T_0|\bigtriangleup b_\mu(s;{\hat{v}})|ds\right|^{8}}\bigg)\times\bigg(\sqrt{\mathbb{E}|\int^T_0|\bigtriangleup b(s;{\hat{v}})|ds|^{8}}\nonumber\\
 &&+\sqrt{\mathbb{E}\left|\int^T_0|\bigtriangleup b_x(s;{\hat{v}})|ds\right|^{8}}+\sqrt{\mathbb{E}\tilde{\mathbb{E}}\left|\int^T_0|\bigtriangleup b_\mu(s;{\hat{v}})|ds\right|^{8}}\nonumber\\
 &&+ \sqrt{\mathbb{E}\int^T_0\left|\bigtriangleup b(s;{\hat{v}})\right|^{8}ds}+\sqrt{\mathbb{E}\int^T_0\left|\bigtriangleup b_x(s;{\hat{v}})\right|^{4}ds}\nonumber\\
 &&+\sqrt{\mathbb{E}\tilde{\mathbb{E}}\int^T_0\left|\bigtriangleup b_\mu(s;{\hat{v}})\right|^{4}ds}+\sqrt{\mathbb{E}\int^T_0 \int^1_0\int^1_0\left| \lambda \bigtriangleup b_{xx}(s;\lambda\eta;{\hat{v}})\right|^{4}d \lambda d\eta ds}\nonumber\\
  &&+\sqrt{\mathbb{E}\tilde{\mathbb{E}}\bar{\mathbb{E}}\int^T_0 \int^1_0\int^1_0\left| \lambda \bigtriangleup b_{\mu\mu}(s;\lambda\eta;{\hat{v}})\right|^{4}d \lambda d\eta ds}\nonumber\\
  &&+\sqrt{\mathbb{E}\tilde{\mathbb{E}}\bar{\mathbb{E}}\int^T_0 \int^1_0\int^1_0 \left|\lambda \bigtriangleup b_{\mu y}(s;\lambda\eta;{\hat{v}})\right|^{4}d \lambda d\eta ds}\nonumber\\
    &&+\sqrt{\mathbb{E}\tilde{\mathbb{E}}\int^T_0 \int^1_0 \left|\lambda \bigtriangleup b_{x\mu}(s;\lambda;{\hat{v}})\right|^{4}d \lambda ds} \bigg)=o\left(|I_{\rho}|^4\right),\  \mbox{as}\ |I_{\rho}|\rightarrow 0.
      \end{eqnarray}
      Similarly
   \begin{eqnarray}
&&\mathbb{E} \int^T_0 |B(s;{\hat{v}})|^2ds=o\left(|I_{\rho}|^4\right),\  \mbox{as}\ |I_{\rho}|\rightarrow 0.
    \end{eqnarray}
    
 \indent Finally, by (\ref{eq48})
  we have the desire result.
\end{proof}

\section{Expansion of the cost functional with respect to control variable}
\indent In this section, we use the method of Lemma \ref{lm43} again to study the expansion of the cost functional according to different order of the purtubation.
\begin{lemma}\label{lm51}
Assume that Hypothesis \ref{H1} holds. Define
 \begin{eqnarray}
{J^{*}}(v_{\cdot})&=&J(v_{\cdot})-J(u_{\cdot})-\mathbb{E}\int^t_0 \bigg\{ h_x(s)X^{v,1}_s+\tilde{\mathbb{E}}\left[h_{\mu}(s) \tilde{X}^{v,1}_s\right]+\bigtriangleup h(s,v)\nonumber\\
  &&+h_x(s)X^{v,2}_s+\tilde{\mathbb{E}}\left[h_{\mu}(s) \tilde{X}^{v,2}_s\right]+\bigtriangleup h_x(s,v)X^{v,1}_s+\tilde{\mathbb{E}}\left[\bigtriangleup h_{\mu}(s,v)\tilde{X}^{v,1}_s\right] \nonumber\\
 &&+\frac{1}{2} h_{xx}(s)X^{v,1}_s\times X^{v,1}_s+ \tilde{\mathbb{E}}\left[ h_{x \mu}(s)X^{v,1}_s\times \tilde{X}^{v,1}_s\right]\nonumber\\
 &&+\frac{1}{2} \tilde{\mathbb{E}}\left[ h_{y \mu}(s)\tilde{X}^{v,1}_s\times \tilde{X}^{v,1}_s\right]+\frac{1}{2}\bar{\mathbb{E}} \tilde{\mathbb{E}}\left[ h_{\mu \mu}(s)\tilde{X}^{v,1}_s\times \bar{X}^{v,1}_s\right] \bigg\}ds.
 \end{eqnarray}
 Recall that $\hat{v}$ is defined by (\ref{eq49}) by $I_\rho$. Then,  when  $|I_{\rho}|\rightarrow 0$, we have
  \begin{eqnarray}
&J^*({{\hat{v}}}_{\cdot})=o\left(|I_{\rho}|^{2}\right).&
   \end{eqnarray} 
   
\end{lemma}

\begin{proof}
Denote
 \begin{eqnarray}
  Y^v_t=\int^t_0 h(s,X^v_s,P_{X^v_s},v_s)ds.
   \end{eqnarray}
 By (\ref{eq43}) and Lemma 4.3, we have
  \begin{eqnarray}
  Y^v_t=Y^u_t+Y^{v,1}_t+Y^{v,2}_t+Y^{v*}_t,
    \end{eqnarray}
  where
   \begin{eqnarray}
 Y^v_t&=&\int^t_0 h(s,X^v_s,P_{X^v_s},v_s)ds,\\
 Y^{v,1}_t&=&\int^t_0 \bigg\{ h_x(s)X^{v,1}_s+\tilde{\mathbb{E}}\left[h_{\mu}(s) \tilde{X}^{v,1}_s\right]+\bigtriangleup h(s,v)\bigg\} ds,\\
  Y^{v,2}_t&= &\int^t_0\bigg \{ h_x(s)X^{v,2}_s+\tilde{\mathbb{E}}\left[h_{\mu}(s) \tilde{X}^{v,2}_s\right]+\bigtriangleup h_x(s,v)X^{v,1}_s+\tilde{\mathbb{E}}[\bigtriangleup h_{\mu}(s,v)\tilde{X}^{v,1}_s] \nonumber\\
 &&+\frac{1}{2} h_{xx}(s)X^{v,1}_s\times X^{v,1}_s+ \tilde{\mathbb{E}}\left[ h_{x \mu}(s)X^{v,1}_s\times \tilde{X}^{v,1}_s\right]\nonumber\\
 &&+\frac{1}{2} \tilde{\mathbb{E}}\left[ h_{y \mu}(s)\tilde{X}^{v,1}_s\times \tilde{X}^{v,1}_s]+\frac{1}{2}\bar{\mathbb{E}} \tilde{\mathbb{E}}[ h_{\mu \mu}(s)\tilde{X}^{v,1}_s\times \bar{X}^{v,1}_s\right] \bigg\}ds.
      \end{eqnarray}
 Then,
 \begin{eqnarray}
  J(v_{\cdot})-J(u_{\cdot})=\mathbb{E}Y^{v}_T-\mathbb{E}Y^{u}_T,
   \end{eqnarray}
   and hence,
   \begin{eqnarray}
 J^*(\hat{v}_{\cdot})=\mathbb{E}Y^{v}_T-\mathbb{E}Y^{u}_T-\mathbb{E}Y^{v,1}_T-\mathbb{E}Y^{v,2}_T.
   \end{eqnarray}
Using the same method in Lemma \ref{lm43}, we complete the proof.
\end{proof}

\indent Now, we proceed to deriving the expansion of the perturbed cost function.
\begin{eqnarray}
 X^{v,12}_t&=&\int^t_0 \bigg\{ b_x(s)X^{v,1}_s+\tilde{\mathbb{E}}\left[b_{\mu}(s) \tilde{X}^{v,1}_s\right]+\bigtriangleup b(s,v)\nonumber\\
  &&+b_x(s)X^{v,2}_s+\tilde{\mathbb{E}}\left[b_{\mu}(s) \tilde{X}^{v,2}_s\right]+\bigtriangleup b_x(s,v)X^{v,1}_s+\tilde{\mathbb{E}}\left[\bigtriangleup b_{\mu}(s,v)\tilde{X}^{v,1}_s\right] \nonumber\\
 &&+\frac{1}{2} b_{xx}(s)X^{v,1}_s\times X^{v,1}_s+ \tilde{\mathbb{E}}\left[ b_{x \mu}(s)X^{v,1}_s\times \tilde{X}^{v,1}_s\right]\nonumber\\
 &&+\frac{1}{2} \tilde{\mathbb{E}}\left[ b_{y \mu}(s)\tilde{X}^{v,1}_s\times \tilde{X}^{v,1}_s\right]+\frac{1}{2}\bar{\mathbb{E}} \tilde{\mathbb{E}}\left[ b_{\mu \mu}(s)\tilde{X}^{v,1}_s\times \bar{X}^{v,1}_s\right] \bigg\}ds\nonumber\\
  &&+\int^t_0 \bigg\{ \sigma_x(s)X^{v,1}_s+\tilde{\mathbb{E}}\left[\sigma_{\mu}(s) \tilde{X}^{v,1}_s\right] + \sigma_x(s)X^{v,2}_s+\tilde{\mathbb{E}}\left[\sigma_{\mu}(s) \tilde{X}^{v,2}_s\right]\nonumber\\
 &&+\frac{1}{2} \sigma_{xx}(s)X^{v,1}_s\times X^{v,1}_s+\tilde{\mathbb{E}}\left[ \sigma_{x \mu}(s)X^{v,1}_s\times \tilde{X}^{v,1}_s\right]\nonumber\\
 &&+\frac{1}{2}\tilde{\mathbb{E}}\left[ \sigma_{y \mu}(s)\tilde{X}^{v,1}_s\times \tilde{X}^{v,1}_s\right]+\frac{1}{2}\bar{\mathbb{E}} \tilde{\mathbb{E}}\left[ \sigma_{\mu \mu}(s)\tilde{X}^{v,1}_s\times \bar{X}^{v,1}_s\right] \bigg\}dB_s.
 \end{eqnarray}

Recalling that $p_t$ is given by (\ref{eq210d}) and applying It\^o's formula to $p_t X^{v,12}_t$, we obtain
\begin{eqnarray}
&&\mathbb{E}\int^T_0 \left \{ h_x(s)X^{v,12}_s + \tilde{\mathbb{E}}\left[h_{\mu}(s) \tilde{X}^{v,12}_s\right] \right\}ds \nonumber\\
&=&\mathbb{E }\int^T_0  p(s)\bigg\{ \bigtriangleup b(s;v)+\bigtriangleup b_x(s;v)X^{v,1}_s+\tilde{\mathbb{E}}\left[\bigtriangleup b_{\mu}(s;v)\tilde{X}^{v,1}_s\right]\nonumber\\
&&+\frac{1}{2} b_{xx}(s)X^{v,1}_s\times X^{v,1}_s+ \tilde{\mathbb{E}}\left[ b_{x \mu}(s)X^{v,1}_s\times \tilde{X}^{v,1}_s\right]\nonumber\\
 &&+\frac{1}{2} \tilde{\mathbb{E}}\left[ b_{y \mu}(s)\tilde{X}^{v,1}_s\times \tilde{X}^{v,1}_s\right]+\frac{1}{2}\bar{\mathbb{E}} \tilde{\mathbb{E}}\left[ b_{\mu \mu}(s)\tilde{X}^{v,1}_s\times \bar{X}^{v,1}_s\right]\bigg \}ds\nonumber\\
 &&+\mathbb{E}\int^T_0  q(s)\bigg\{ \frac{1}{2} \sigma_{xx}(s)X^{v,1}_s\times X^{v,1}_s+\tilde{\mathbb{E}}\left[ \sigma_{x \mu}(s)X^{v,1}_s\times \tilde{X}^{v,1}_s\right]\nonumber\\
 &&+\frac{1}{2}\tilde{\mathbb{E}}\left[ \sigma_{y \mu}(s)\tilde{X}^{v,1}_s\times \tilde{X}^{v,1}_s\right]+\frac{1}{2}\bar{\mathbb{E}} \tilde{\mathbb{E}}\left[ \sigma_{\mu \mu}(s)\tilde{X}^{v,1}_s\times \bar{X}^{v,1}_s\right]\bigg\}ds.
  \end{eqnarray}
Hence,
\begin{eqnarray}\label{eq51}
&&J(v_{\cdot})-J(u_{\cdot})\nonumber\\
&=&\mathbb{E}\int^T_0 \bigtriangleup H(s;v)ds+\mathbb{E}\int^T_0 \bigtriangleup H_x(s;v)X^{v,1}_sds\nonumber\\
&&+\mathbb{E}\tilde{\mathbb{E}}\int^T_0  \bigtriangleup H_\mu(s;v)\tilde{X}^{v,1}_sds\nonumber\\
&&+\frac{1}{2} \mathbb{E} \int^T_0 Trace\left \{ H_{xx}(s)X^{v,1}_s X^{*v,1}_s\right \} ds\nonumber\\
&&+ \mathbb{E}\tilde{\mathbb{E}} \int^T_0 Trace\left\{ H_{x\mu}(s)X^{v,1}_s  \tilde{X}^{*v,1}_s\right\} ds\nonumber\\
&&+\frac{1}{2} \mathbb{E}\tilde{\mathbb{E}} \int^T_0 Trace\left\{  H_{y\mu}(s)\tilde{X}^{v,1}_s \tilde{X}^{*v,1}_s \right\}ds\nonumber\\
&&+\frac{1}{2} \mathbb{E}\tilde{\mathbb{E}}\bar{\mathbb{E}} \int^T_0 Trace\left \{  H_{\mu\mu}(s)\tilde{X}^{v,1}_s {\bar{X}^{*v,1}_s}\right\} ds
+J^{*}\left(v(\cdot)\right).
 \end{eqnarray}

 \indent Now we apply the range theorem for vector-valued measures due to \cite{MY2007},  to deduce the variational inequality.\\ 
 \indent Recall that ${\hat{v}}_t$ is defined by (\ref{eq49}). According to Lemma 4.1 \cite{MY2007}, for any $0<\rho<1$, we can choose a suitable $I_{\rho}\subset [0,T]$ such that
$|I_{\rho}|=\rho T ,$
\begin{eqnarray}\label{eq58}
\rho\int^T_0\bigtriangleup b(s;v)ds=\int_{I_{\rho}}\bigtriangleup b(s;v)ds+\eta^*, \ 
\mathbb{E}|\eta^*|^2=o(\rho^4),
\end{eqnarray}

\begin{eqnarray}\label{eq513}
\rho\int^T_0\mathbb{E}\big [\bigtriangleup H(s;v)\big ]ds=\int_{ I_{\rho}}\mathbb{E}\big[\bigtriangleup H(s;v)\big]ds+o(\rho^2),
\end{eqnarray}
and
\begin{eqnarray}\label{eq53}
&&\rho\int^T_0\mathbb{E}\bigg \{ \bigtriangleup H_x(s;v)X^{v,1}_s+\tilde{\mathbb{E}} \left[ \bigtriangleup H_\mu(s;v)\tilde{X}^{v,1}_s\right]+\bigtriangleup {b}^*(s;v) P_s{X}^{v,1}_s \bigg\}ds\nonumber\\
&=&\int_{ I_{\rho}}\mathbb{E}\bigg \{ \bigtriangleup H_x(s;v)X^{v,1}_s+\tilde{\mathbb{E}} \left[ \bigtriangleup H_\mu(s;v)\tilde{X}^{v,1}_s\right]+\bigtriangleup {b}^*(s;v) P_s{X}^{v,1}_s \bigg\}ds\nonumber\\
&&+o(\rho^2).
\end{eqnarray}

\begin{lemma}\label{lm52}
For the $I_{\rho}$ above, $t\in[0,T]$, we also have
\begin{eqnarray}
&&\rho\int^t_0\bigtriangleup b(s;v)ds=\int_{I_{\rho}\cap[0,t]}\bigtriangleup b(s;v)ds+\eta_t^*, 
\ \sup_{0\leq t\leq T}\mathbb{E}|\eta^*_t|^2=o(\rho^4).
\end{eqnarray}
\end{lemma}

\indent The proof of the above lemma is essentially the same as  Lemma 4.1  \cite{MY2007}. For the convenience of readers, we present the proof here.

\begin{proof}
Let $\phi_i(\cdot)\in L^{2}\left(\Omega; L^{2}(0,T;\mathbb{R}^{l_i})\right),$ $l_i\geq 1$, $i=1,\cdots,k.$ Suppose 
$$\sup_{0\leq t \leq T}\mathbb{E}\left|\phi_1(t)\right|^{2}< \infty$$.
 Given $0<\rho<1$ and set $\delta=\rho^{5}$, then there exists a $n>0$, we can find a process $\phi^\rho_i(\cdot)$ in the form of
$$\phi^{\rho}_i(t)=\sum_{j=0}^n\xi_i^jI_{[t_j,t_{j+1})}(t), \ \ \ 1\leq i\leq k,$$
with $0=t_0<t_1<\cdots<t_{n+1}=T,$ $\max|t_{i+1}-t_i|< \delta$, $\xi^j_i$ being $\mathcal{F}_{t_j}$-measurable, s.t.
\begin{eqnarray}
&&\sup_{0\leq t \leq T}\mathbb{E}\left|\phi_1(t)-\phi^{\rho}_1(t)\right|^{2}< \delta.\\
&&\sum^k_{i=2}\mathbb{E}\left(\int^T_0\left|\phi_i(t)-\phi^{\rho}_i(t)\right|^{2}dt\right)< \delta.
\end{eqnarray}
Note that we can always choose the partition $\{t_j\}_{0 \leq j \leq n+1}$ independent of $i=1,\cdots, k.$ Now letting 
$$G=\bigcup^n_{j=0}\left[t_j,t_j+\rho(t_{j+1}-t_j)\right).$$
It's easy to see $|G|=\rho T$. Thus (\ref{eq58}), (\ref{eq513}), and  (\ref{eq53}) are proved by taking $\phi_i$ suitably. For any $s\in[0,T]$, we can always find a $m\geq 0$, s.t. $s\in [t_m,t_{m+1}).$ Then we have 
\begin{eqnarray}
&&\sup_{0\leq s \leq T}\mathbb{E}\left|\int^T_0I_{\{t\leq s\}}\left(1-\frac{I_G(t)}{\rho}\right)\phi^{\rho}_1(t) dt\right|^2\nonumber\\
&\leq&\sup_{0\leq s \leq T}\mathbb{E}\left|\sum^{m-1}_{j=0}\xi^j_1\left[(t_{j+1}-t_j)-\frac{\rho(t_{j+1}-t_j)}{\rho}\right]+\xi^m_1\left[(s-t_m)-\frac{(s-t_m)}{\rho}\right]\right|^2\nonumber\\
&&+\sup_{0\leq s \leq T}\mathbb{E}\bigg|\sum^{m-1}_{j=0}\xi^j_1\left[(t_{j+1}-t_j)-\frac{\rho(t_{j+1}-t_j)}{\rho}\right]\nonumber\\
&&+\xi^m_1\left[(s-t_m)-\frac{\rho(t_{m+1}-t_m)}{\rho}\right]\bigg|^2\nonumber\\
&\leq& K\delta^2 \rho^{-2}+ K\delta^2,
\end{eqnarray}
where $I_G(t)$ is the indicator function of $G$. Thus, for each $1\leq i\leq k,$
\begin{eqnarray}
&&\sup_{0\leq s\leq T}\mathbb{E}\left|\int^T_0 I_{\{t\leq s\}}\left(1-\frac{I_G(t)}{\rho}\right)\phi_1(t) dt\right|^{2}\nonumber\\
&\leq&\sup_{0\leq s\leq T}\mathbb{E}\left|\int^T_0I_{\{t\leq s\}}\left(1-\frac{I_G(t)}{\rho}\right)\phi_1(t)-I_{\{t\leq s\}}\left(1-\frac{I_G(t)}{\rho}\right)\phi^{\rho}_1(t) dt\right|^{2}\nonumber\\\
&&+K\delta^{2}\rho^{-2}\nonumber\\
&\leq &\mathbb{E}\left|\int^T_0\left|\left(1-\frac{I_G(t)}{\rho}\right)\left[\phi_1(t)-\phi^{\rho}_1(t)\right]\right| dt\right|^{2}+K\delta^{2}\rho^{-2}\nonumber\\
&\leq& \rho^{-2}T\delta+K\delta^{2}\rho^{-2}.
\end{eqnarray}
Set $\phi_1(t)=\bigtriangleup b(t,v)$, then letting $\rho\rightarrow 0$, we finish the proof.
\end{proof}

\begin{lemma}\label{lm53}
For any $\rho\in[0,1]$, there exists a subset $I_\rho$ of $[0,T]$, such that
\begin{eqnarray}\label{eq54}
&&\lim_{\rho\rightarrow 0+}\sup_{0\leq t \leq T}\mathbb{E}\bigg[\bigg\vert\frac{X^{{\hat{v}},1}_t-\rho X^{v,1}_t}{\rho}\bigg\vert^2\bigg]=0.
\end{eqnarray}
where $\hat{v}$ is defined by (\ref{eq49}) with $I_\rho$ given here.
\end{lemma}

\begin{proof}
By (\ref{eq41}) and lemma \ref{lm52}, we have
\begin{eqnarray}
&&X^{\hat{v},1}_t-\rho X^{v,1}_t\nonumber\\
&=&\int^t_0 b_x(s)\left(X^{\hat{v},1}_t-\rho X^{v,1}_t\right)ds+\int^t_0\tilde{\mathbb{E}}\left[b_\mu(s)\left(\tilde{X}^{\hat{v},1}_t-\rho\tilde{X}^{v,1}_t\right)\right]ds\nonumber\\
&&+\int^t_0\left(\bigtriangleup b(s;\hat{v})-\rho\bigtriangleup b(s;{v})\right)ds\nonumber\\
&&+\int^t_0 \sigma_x(s)\left(X^{\hat{v},1}_t-\rho X^{v,1}_t\right)ds+\int^t_0\tilde{\mathbb{E}}\left[\sigma_\mu(s)\left(\tilde{X}^{\hat{v},1}_t-\rho\tilde{X}^{v,1}_t\right)\right]dB_s.
\end{eqnarray}
Thus, we can get
\begin{eqnarray}
&&\mathbb{E} \left| X^{\hat{v},1}_t-\rho X^{\hat{v},1}_t \right|^2 \nonumber\\
&\leq&K\int^t_0 \mathbb{E} \left| X^{\hat{v},1}_t-\rho X^{\hat{v},1}_t \right|^2ds+\sup_{0\leq t\leq T}\mathbb{E}\left[\int^t_0\left(\bigtriangleup b(s;\hat{v})-\rho\bigtriangleup b(s;{v})\right)ds\right]^2.\nonumber\\
\end{eqnarray}
Notice that 
\begin{eqnarray}
&&K\sup_{0\leq t\leq T}\mathbb{E}\frac{1}{\rho^2}\left[\int^t_0{\left(\bigtriangleup b(s;\hat{v})-\rho\bigtriangleup b(s;{v})\right)}ds\right]^2\nonumber\\
&=&K \sup_{0\leq t \leq T}\frac{1}{\rho^2}\mathbb{E}|\eta^*_t|^2=o(1).
\end{eqnarray}
 By Gronwall's inequality, we get (\ref{eq54}). 
\end{proof}

\indent Further more, inspired by \cite{YZ2000}, we also have the following lemma holds.

\begin{lemma}\label{lm54}
For any $\rho\in[0,1]$, there exists a set $I_\rho\in [0,T]$ and a matrix value process $\Phi(t)$, s.t.  $X^{\hat{v},1}_t$ is represented by the following
\begin{eqnarray}
X^{\hat{v},1}_t=\Phi(t)\int^t_0\Phi^{-1}(s)\bigtriangleup b(s;\hat{v})dt+A^*_t,
\end{eqnarray}
where
\begin{eqnarray}
&&\lim_{\rho\rightarrow 0+}\sup_{0\leq t \leq T}\mathbb{E}\bigg[\bigg\vert\frac{A^*_t}{\rho}\bigg\vert^2\bigg]=0.
\end{eqnarray}
\end{lemma}

\begin{proof}
Let $\Phi(t)$ be the unique solution of the following matrix value SDE:
\begin{eqnarray}
\Phi(t)=I+\int^t_0\left\{b_x(s)+\tilde{\mathbb{E}}\left[b_\mu(s)\right]\right\}\Phi(s)ds+\int^t_0\left\{\sigma_x(s)+\tilde{\mathbb{E}}\left[\sigma_\mu(s)\right]\right\}\Phi(s)dB_s,
\end{eqnarray}
and set
\begin{eqnarray}
\Psi(t)&=&I+\int^t_0 \Psi(s)\left\{-\left(b_x(s)+\tilde{\mathbb{E}}\left[b_\mu(s)\right]\right)+\left(\sigma_x(s)+\tilde{\mathbb{E}}\left[\sigma_\mu(s)\right]\right)^2\right\}ds\nonumber\\
&&- \int^t_0 \Psi(s)\left\{\sigma_x(s)+\tilde{\mathbb{E}}\left[\sigma_\mu(s)\right]\right\}dB_s.
\end{eqnarray}
Applying It\^o's formula to $[\Psi_t\Phi_t]$, we can easily get $d[\Psi_t\Phi_t]=0$, which means $\Psi_t\Phi_t\equiv I$, i.e. $\Psi_t=\Phi_t^{-1}.$ Applying It\^o's formula to $d\left[\Psi_tX^{\hat{v},1}_t\right]$, by Lemma \ref{lm53}, we then get our desire result.
\end{proof}

\indent Now we continue to derive the expansion. Applying It\^o's formula to $Y_{t}:=X^{\hat{v},1}_{t}{X}^{*\hat{v},1}_{t}$, we have
\begin{eqnarray}
dY_t&=&\bigg \{ Y_t{b}_x^*(t)+{X}^{\hat{v},1}_t\tilde{\mathbb{E}}\left[\tilde{X}^{*\hat{v},1}_t {b}_{\mu}^*(t)\right]+{X}^{\hat{v},1}_t\bigtriangleup {b}^*(t;{\hat{v}}) \nonumber\\
&&+b_x(t)Y_t+\tilde{\mathbb{E}}\left[b_{\mu}(t)\tilde{X}^{\hat{v},1}_t\right]{X}^{*\hat{v},1}_t
+\bigtriangleup b(t;v){X}^{*\hat{v},1}_t \\
&&+\left(\sigma_x(t)X^{\hat{v},1}_t+\tilde{\mathbb{E}}\left[\sigma_\mu(t)X^{\hat{v},1}_t\right]\right)\left(X^{*\hat{v},1}_t\sigma^*_x(t)+\tilde{\mathbb{E}}\left[X^{*\hat{v},1}_t\sigma^*_\mu(t)\right]\right)\bigg \}dt \nonumber\\
&&+\bigg \{ Y_t{\sigma}_x^*(t)+{X}^{\hat{v},1}_t\tilde{\mathbb{E}}\left[\tilde{X}^{*\hat{v},1}_t {\sigma}_{\mu}^*(t)\right]+\sigma_x(t)Y_t+\tilde{\mathbb{E}}\left[\sigma_{\mu}(t)\tilde{X}^{\hat{v},1}_t\right]{X}^{*\hat{v},1}_t\bigg \}dB_t\nonumber.
\end{eqnarray}

\indent  Applying It\^o's formula to $P_tY_t$, according to (\ref{eq54}), we obtain
\begin{eqnarray}\label{eq52}
&& \mathbb{E} \int^T_0 Trace\big\{H_{xx}(s)X^{\hat{v},1}_s X^{*\hat{v},1}_s\big\} ds
+ 2\mathbb{E}\tilde{\mathbb{E}} \int^T_0 Trace\big\{  H_{x\mu}(s)X^{\hat{v},1}_s  \tilde{X}^{*\hat{v},1}_s\big\} ds\nonumber\\
&&+\mathbb{E}\tilde{\mathbb{E}} \int^T_0 Trace\big\{  H_{y\mu}(s)\tilde{X}^{\hat{v},1}_s \tilde{X}^{*\hat{v},1}_s \big\}ds\nonumber\\
&&+\mathbb{E}\tilde{\mathbb{E}}\bar{\mathbb{E}} \int^T_0 Trace \big\{  H_{\mu\mu}(s)\tilde{X}^{\hat{v},1}_s {\bar{X}^{*\hat{v},1}_s}\big\} ds\nonumber\\
&=&\mathbb{E}\int^T_0 Trace\big\{P_s{X}^{\hat{v},1}_s\bigtriangleup {b}^*(t;{\hat{v}})
+P_s\bigtriangleup b(s;\hat{v}){X}^{*\hat{v},1}_s\big\}ds+o(\rho^2)\nonumber\\
&=&2\mathbb{E}\int^T_0\left[\bigtriangleup {b}^*(s;\hat{v}) P_s {X}^{\hat{v},1}_s\right]ds+o(\rho^2).
\end{eqnarray}

\indent Putting (\ref{eq52}) into (\ref{eq51}), we conclude
\begin{eqnarray}\label{eq56}
J(\hat{v}_{\cdot})-J(u_{\cdot})&=&\mathbb{E}\int^T_0 \bigtriangleup H(s;\hat{v})ds+\mathbb{E}\int^T_0 \bigtriangleup H_x(s;\hat{v})X^{\hat{v},1}_sds\nonumber\\
&&+\mathbb{E}\tilde{\mathbb{E}}\int^T_0  \bigtriangleup H_\mu(s;\hat{v})\tilde{X}^{\hat{v},1}_sds\nonumber\\
&&\mathbb{E}\int^T_0\left[\bigtriangleup {b}^*(s;\hat{v}) P_s{X}^{\hat{v},1}_s\right]ds+o(\rho^2).
\end{eqnarray}

From Lemma 5.1, (\ref{eq54}) and (\ref{eq56}), we obtain
\begin{eqnarray}
J({\hat{v}}_{\cdot})-J(u_{\cdot})
&=&\mathbb{E}\int^T_0 \bigtriangleup H(s;{\hat{v}})ds+\rho\mathbb{E}\int^T_0 \bigtriangleup H_x(s;{\hat{v}})X^{{{v}},1}_sds\nonumber\\
&&+\rho\mathbb{E}\tilde{\mathbb{E}}\int^T_0  \bigtriangleup H_\mu(s;{\hat{v}})\tilde{X}^{{{v}},1}_sds+\rho\mathbb{E}\int^T_0\left[\bigtriangleup {b}^*(s;{\hat{v}}) P_s{X}^{{{v}},1}_s\right]ds\nonumber\\
&&+o(\rho^2)\nonumber\\
&=&\mathbb{E}\int_{I_{\rho}} \bigtriangleup H(s;{{v}})ds+\rho\mathbb{E}\int_{I_{\rho}} \bigtriangleup H_x(s;{{v}})X^{{{v}},1}_sds\nonumber\\
&&+\rho\mathbb{E}\tilde{\mathbb{E}}\int_{I_{\rho}}  \bigtriangleup H_\mu(s;{{v}})\tilde{X}^{{{v}},1}_sds+\rho\mathbb{E}\int_{I_{\rho}}\left[\bigtriangleup {b}^*(s;{{v}}) P_s{X}^{{{v}},1}_s\right]ds\nonumber\\
&&+o(\rho^2).
\end{eqnarray}
 Finally, according to (\ref{eq53}), we have
\begin{eqnarray}\label{eq55}
&&J({\hat{v}}_{\cdot})-J(u_{\cdot})\nonumber\\
&=&\rho\mathbb{E}\int_0^T \bigtriangleup H(s;{{v}})ds+\rho^2\mathbb{E}\int_0^T \bigtriangleup H_x(s;{{v}})X^{{{v}},1}_sds\nonumber\\
&&+\rho^2\mathbb{E}\tilde{\mathbb{E}}\int_0^T  \bigtriangleup H_\mu(s;{{v}})\tilde{X}^{{{v}},1}_sds+\rho^2\mathbb{E}\int_0^T\left[\bigtriangleup {b}^*(s;{{v}}) P_s{X}^{{{v}},1}_s\right]ds\nonumber\\
&&+o(\rho^2).
\end{eqnarray}

\section{The proofs of the stochastic maximum principle.}
\indent Although the first-order SMP has been obtained by \cite{BLM2016}, we give a proof here for completeness. In fact, after the preparation of the previous sections which will also be needed in the proof of the second-order SMP, this proof does not take too much extra effort.\\
\indent \textbf{Proof of first-order SMP:} Since $\big(X^u_{\cdot}, u_{\cdot}\big)$ is an optimal pair of our system, it follows from (\ref{eq55}) that
\begin{eqnarray}
J({\hat{v}}_{\cdot})-J(u_{\cdot})=\rho\mathbb{E}\int_0^T \bigtriangleup H(s;v)ds+o(\rho)\geq 0.
\end{eqnarray}
 for any $\rho\in [0,T]$, $\forall v(\cdot)\in \mathcal{U}$.
 Setting $\rho \rightarrow 0+$, we obtain
 \begin{eqnarray}\label{eq63}
\mathbb{E}\int_0^T \bigtriangleup H(s;v)ds\geq 0, \ \ \forall v(\cdot)\in \mathcal{U}.
\end{eqnarray}
  Then we can deduce that, for any fixed $v\in \mathcal{U}$, there exists a null subset $S^v\subset [0,T]\times\Omega$, such that for each $(t,\omega)\in \left(S^v\right)^c$,
\begin{eqnarray}\label{eq62}
 \bigtriangleup H(s;v)\geq 0.
\end{eqnarray}
Otherwise, suppose that
$$A=\left\{(s,\omega): \bigtriangleup H(s;v^*)<0\right\}$$
has positive measure in $[0,T]\times \Omega$, for a $v^*\in \mathcal{U}.$ Let
$$\hat{v}^*=v^*\text{1}_A+u\text{1}_{A^c}.$$
Then,
 \begin{eqnarray}
\mathbb{E}\int^T_0 \bigtriangleup H(s;\hat{v}^*)ds =\mathbb{E}\int_0^t \bigtriangleup H(s;{v}^*)\text{1}_A ds<0.
\end{eqnarray}
This contradicts from (\ref{eq63}).\\
\indent Select a countable dense subset $\{v_s^{(i)}\}^\infty_{i=1}\subset U,$ set
 $$S_0=\bigcup^\infty_{i=1} S^{v^{(i)}}.$$
 Then, $S_0$ is a null subset of $[0,r]\times\Omega$, and for $(t,\omega)\in S:= \left(S_0\right)^c$, we get
 \begin{eqnarray}\label{eq61}
 \bigtriangleup H(s;v^{(i)})\geq 0.
\end{eqnarray}
By Fubini's theorem, it is easy to see that there exists a null subset $T_0$ of $[0,T]$, such that $\forall t\in T_0^c,$ (\ref{eq61}) holds $a.s.$.\\
\indent Finally, from the continuity of the function and the denseness of $\{v^{(i)}\}^\infty_{i=1}$, we have for $t\in\left(T_0\right)^c$,
 \begin{eqnarray}
 \bigtriangleup H(s;v)\geq 0,\ \forall v\in U, a.s..
\end{eqnarray}
\qed

\indent Now, we proceed to presenting the proof of the second-order stochastic maximum principle for singular generalized mean-field control problem. \\
{\em Proof of Theorem \ref{thm0209a}}.
 The optimality and the singularity imply that
$$\bigtriangleup H(t;v)\equiv 0, \ \forall v\in V.$$
According to (\ref{eq55}), we have 
 \begin{eqnarray}\label{eq64}
 &&\mathbb{E}\int^{t_2}_{t_1}\bigg \{\bigtriangleup H_x(s;v)X^{v,1}_s+\tilde{\mathbb{E}}\left[\bigtriangleup H_\mu(s;v)\tilde{X}^{v,1}_s\right]+\bigtriangleup {b}^*(s;v) P_s{X}^{v,1}_s\bigg\}ds\nonumber\\
 &\geq& 0,\ \ \ \forall v\in \mathcal{V}(t_1,t_2), a.s.,
 \end{eqnarray}
 where
 \begin{eqnarray*}
 \mathcal{V}(t_1,t_2):=&&\Big\{v(\cdot)\in \mathcal{U}|v_t\in V, a.s., a.e., t\in[t_1,t_2]; v(t)=u(t), \\
 &&t\in[0,T]\setminus [t_1,t_2] \Big\}.
\end{eqnarray*}

\indent  As in \cite{H1976}\cite{T2010}, denote by $\{ t_i \}^\infty_{i=1}$ the collection of all rational numbers in $[0,T],$ and $\{v_k\}^\infty_{k=1}$ a dense subset of $V$. Because of the fact that $\mathcal{F}_t$ is countability generated for $t\in[0,T],$ we can assume $\{A_{i,j}\}^\infty_{j=1}$ generates $\mathcal{F}_{t_{i}}, \ i=1,2,\cdots.$ For any $\tau\in [t_i,T)$ and $\theta\in (0,T-\tau),$ write $E^i_{\theta}=[\tau,\tau+\theta),$ and define
\begin{equation}
v^{k}_{i,j}(t,\omega)=\left\{
\begin{array}{ccl}
v_k(t,\omega), \ \ \ \ \ \ \  (t,\omega)\in E^i_\theta \times A_{i,j},\\
\ \  \ u(t,\omega), \ \ \ \ \ \ \  (t,\omega)\in \left(E^i_\theta \times A_{i,j}\right)^c.
  \end{array}
  \right.
\end{equation}
Let $X^{1k}_{ij}$ be the solution to the equation (\ref{eq41}) with respect to $v^k_{i,j}(\cdot)$. Notice that we can always choose suitable $I_\theta$, such that $I_\theta \cap E^i_\theta= E^i_\theta.$ So Lemma \ref{lm54} holds for $X^{1k}_{ij}$.

By Lemma 4.1 \cite{ZZ2015}, lemma \ref{lm42} and Lebesgue differential theorem, there is a null subset $T^k_{ij}\subset [0,T]$ such that for $\tau\in \left(T^k_{ij}\right)^c$, we have
\begin{eqnarray}\label{eq65}
0&\leq&\lim_{\theta\rightarrow 0+}\frac{1}{\theta^2}\int^{\tau+\theta}_{\tau} \mathbb{E}\bigg \{\bigtriangleup H_x(s;v^{k}_{i,j})X^{1k}_{ij}(s)+\tilde{\mathbb{E}}\left[\bigtriangleup H_\mu(s;v^{k}_{i,j})\tilde{X}^{1k}_{ij}(s)\right]\nonumber\\
&&+\bigtriangleup {b}^*(s;v^{k}_{i,j}) P_s X^{1k}_{ij}(s)\bigg\}ds\nonumber\\
&=&\lim_{\theta\rightarrow 0+}\frac{1}{\theta^2}\int^{\tau+\theta}_{\tau} \mathbb{E}\bigg \{\bigtriangleup H_x(s;v^{k}_{i,j})\Phi(s)\int^s_\tau \Phi^{-1}(r) \bigtriangleup b(r;v_{k})\text{1}_{A_{ij}}dr\nonumber\\
&&+\tilde{\mathbb{E}}\left[\bigtriangleup H_\mu(s;v^{k}_{i,j})\Phi(s)\int^s_\tau \Phi^{-1}(r) \bigtriangleup \tilde{b}(r;v_{k})\text{1}_{A_{ij}}dr\right]\nonumber\\
&&+\bigtriangleup {b}^*(r;v^{k}_{i,j}) P_s \Phi(s)\int^s_\tau \Phi^{-1}(r) \bigtriangleup {b}(r;v_{k})\text{1}_{A_{ij}}dr\bigg\}ds\nonumber\\
&=&\mathbb{E}\bigg \{\bigtriangleup H_x(\tau;v_k)\bigtriangleup b(\tau;v_k)\text{1}_{A_{ij}}+\tilde{\mathbb{E}}\left[\bigtriangleup H_\mu(\tau;v_k)\bigtriangleup \tilde{b}(\tau;v_{k})\text{1}_{A_{ij}}\right]\nonumber\\
&&+\bigtriangleup {b}^*(\tau;v_k) P_\tau\bigtriangleup b(\tau;v_k)\text{1}_{A_{ij}}\bigg \}.
\end{eqnarray}

\indent Set
$$T_0=\bigcup_{1\leq i,j,k\leq \infty}T^k_{i,j}.$$
Then, $T_0$ is a null subset of $[0,T]$. For $s\in [0,T]\setminus T_0$ and $i$, we deduce that
\begin{eqnarray}
&&\mathbb{E}\bigg \{\bigtriangleup H_x(s;v_k)\bigtriangleup b(s;v_k)\text{1}_{A_{ij}}+\tilde{\mathbb{E}}\left[\bigtriangleup H_\mu(s;v_k)\bigtriangleup \tilde{b}(s;v_{k})\text{1}_{A_{ij}}\right]\nonumber\\
&&+\bigtriangleup {b}^*(s;v_k) P_s\bigtriangleup b(s;v_k)\text{1}_{A_{ij}}\bigg \}\geq 0,\ \forall j,k=1,2,\cdots,
\end{eqnarray}
which means
\begin{eqnarray}
&&\mathbb{E}\bigg \{\bigtriangleup H_x(s;v)\bigtriangleup b(s;v)\text{1}_{A}+\tilde{\mathbb{E}}\left[\bigtriangleup H_\mu(s;v)\bigtriangleup \tilde{b}(s;v)\text{1}_{A}\right]\nonumber\\
&&+\bigtriangleup {b}^*(s;v) P(s)\bigtriangleup b(s;v)\text{1}_{A}\bigg\}\geq 0,\ \forall v\in V,\  A\in \mathcal{F}_t.
\end{eqnarray}
\indent By virtue of the continuity of the function and the denseness of $\{v_k\}^\infty_{k=1}$, we finish the proof.\qed

\begin{remark}
 We now come back to Example \ref{em11}. It is not hard to check that the second order adjoint process $(P_t,Q_t)\equiv(1,0).$ Then
\begin{eqnarray*}
 &&\bigtriangleup H_x(t;v)\bigtriangleup b(t;v)+\tilde{\mathbb{E}}\left[\bigtriangleup H_\mu(t;v)\bigtriangleup \tilde{b}(t;v)\right]\nonumber\\
&&+\bigtriangleup {b}^*(t;v) P_t\bigtriangleup b(t;v)=v_t^2\geq 0,\ \forall v\in U, a.s..
\end{eqnarray*}
So we can say $u_t\equiv 0$ is the only candidate for optimal controls.
\end{remark}



\medskip
Received xxxx 20xx; revised xxxx 20xx.
\medskip
\end{spacing}
\end{document}